\documentclass[10pt]{article}
\usepackage{latexsym,amsfonts,amsthm,amsmath,amssymb,cases,authblk}
\usepackage{cite}
\usepackage{enumerate}
\usepackage{graphicx}
\usepackage{color}
\DeclareMathOperator{\esssup}{ess\,sup}
\newtheorem{thm}{Theorem}[section]

\newtheorem{lemma}{Lemma}[section]
\newtheorem{prob}{Problem}[section]
\newtheorem{subprob}{Subroblem}[section]
\newtheorem{assumption}{Assumption}[section]
\newtheorem{remark}{Remark}[section]
\newtheorem{prop}{Proposition}[section]
\newtheorem{defn}{Definition}[section]
\newtheorem{example}{Example}[section]
\allowdisplaybreaks[4]
\newcounter{nextauthor}
\def\mathrm{\mbox}

\numberwithin{remark}{section}

\begin{document}
\title{{\Large \bf Optimal Control and Stabilization for MFSDEs with Multiple Defaults}\thanks{This work was supported by the National Natural Science Foundation of China (12171339),  the grant from Chongqing Technology and Business University (2356004) and the Fundamental Research Funds for the Central Universities (2682023CX071).}}
\author[a]{Zhun Gou}
\author[b]{Nan-jing Huang \thanks{Corresponding author: nanjinghuang@hotmail.com; njhuang@scu.edu.cn}}
\author[c]{Ming-hui Wang}
\author[d]{Jian-hao Kang}
\affil[a]{\small\it College of Mathematics and Statistics, Chongqing Technology and Business University, Chongqing 400067, P.R.China}
\affil[b]{\small\it Department of Mathematics, Sichuan University, Chengdu, Sichuan 610064, P.R. China}
\affil[c]{\small\it  School of Economic Mathematics, Southwestern University of Finance and Economics, Chengdu, Sichuan 610074, P.R. China}
\affil[d]{\small\it School of Mathematics, Southwest Jiaotong University, Chengdu, Sichuan 610031, P.R. China}

\date{}
\maketitle

\begin{center}
\begin{minipage}{5.6in}
\noindent{\bf Abstract.} In this paper, we investigate an optimal control problem governed by the mean-field stochastic differential equation with multiple defaults, which is motivated by the optimal investment problems.  This global optimal control problem is transformed and divided into several optimal control subproblems governed by  mean-field stochastic differential equations with single default. We derive both the sufficient and the necessary maximum principles for such subproblems and then give the existence and uniqueness of solutions to the mean-field stochastic differential equation with multiple defaults and the mean-field backward stochastic differential equations with multiple defaults, respectively. Moreover,  the obtained results are applied to solve an optimal investment problem, in which the wealth process is governed by a mean-field stochastic differential equation with multiple defaults.  Finally, we study both the mean square exponential stability and the  almost sure exponential stability  under some mild conditions for the solution to the mean-field stochastic differential equation with multiple defaults under optimal feedback control.
\end{minipage}
\end{center}

\section{Introduction}
\paragraph{}
The theory of enlargement of filtration, started by Jacod \cite{Jacod1985Grossissement},  Jeulin \cite{Jeulin1980Semi} and  Yor \cite{Yor1985Grossissement} in the 1980s, is a powerful tool to preserve the adaptedness of stochastic processes under a change of filtration. Some interest models have been developed in mathematical finance \cite{Bo2013Stochastic, Di2020on, Jeanblanc2020Characteristics} to study the asymmetry of information, in which different agents have a different level of information and credit default events assumed to occur at a random time $\tau$ in the same financial market. The solutions to corresponding problems under the above models are all obtained by employing the theory of the enlargement of filtration.

Usually, when investors consider credit default events, they assume that these events arrive surprisingly, i.e., $\tau$ is independent of the reference filtration $\mathcal{F}$ (these filtrations always are generated by inside the systems). The standard approach to deal with these default events is mainly based on the theory of enlargement a reference filtration $\mathcal{F}$ by the information of $\tau$, which leads to a new global filtration $\mathfrak{F}$. This approach is called the progressive enlargement of filtration. Besides, the usual hypothesis (H) that any $\mathcal{F}$-martingale remains an $\mathfrak{F}$-martingale is required to ensure the martingale representation in $\mathfrak{F}$.
Based on the above approach and hypothesis, Bachir et al. \cite{Bachir2020Stochastic} transformed the stochastic differential equation (SDE) with default into an SDE driven by an additional martingale under the new filtration. Then, they studied the optimal control problem governed by the new SDE system.
Moreover, Peng and Xu \cite{Peng2009BSDE} studied backwards stochastic differential equations (BSDEs) with the random default time and applied it to investigate default risk.
For more details about the method of progressive enlargement of filtration, we refer to \cite{Aksamit2017Enlargement, Aksamit2019Integral, Calvia2020Risk, Di2020on, Dumitrescu2016BSDEs, Elie2020Large, Shen2013A}.

On the other hand, it is well known that mean-field stochastic differential equation (MFSDE) has attracted much interest from many areas including the field of control theory. The optimal control problem governed by MFSDEs has been studied by numerous researchers (see \cite{agram2018mean, agram2021fokker, agram2019mean, Assouli2023Deep, Dumitrescu2018Stochastic, Li2016controlled, ma2018Infinite, meng2015optimal, moon2021Risk, wang2022maximum, zhang2021stochastic, zhang2016Modeling}). Furthermore, the optimal control problems governed by mean-field stochastic differential equation with a single default (SMFSDE) and the ones of mean-field stochastic differential equation with multiple defaults (MMFSDE) were studied in \cite{Bachir2020Stochastic, gou2021a} and \cite{Pham2010Stochastic}, respectively, in which both the SMFSDE and the MMFSDE are considered in the framework of progressive enlargement of filtration. We also note that the conditional MFSDEs system and its applications have attracted much attention recently (see \cite{buckdahn2023general, Caemona2016A, ma2018Kyle}). This system plays an important role in mean-field optimal control problems with partial information \cite{Buckdahn2017A}, and in mean-field games (MFGs) under asymmetric information (especially in MFGs with major and minor players \cite{Carmona2018Mean1, Carmona2018Mean2}).   Nevertheless, to the best of our knowledge, the existing results including \cite{Buckdahn2017A, Pham2010Stochastic} are almost unapplicable to solve the following practical problem, which is  inspired from \cite{Bachir2020Stochastic}.
\begin{prob}\label{+8}
Consider the wealth process with default
\begin{equation*}
\begin{cases}
dX(t)=X(t-)\left[(\alpha-u(t))dt+\beta dB_t+\sum \limits_{i=1}^{2}\varepsilon_id\mathbb{I}_{\tau_i\leq t}\right]+\delta\overline{X}(t)dt, \quad t\in(0,T],\\
X(0)=x_0\in \mathbb{R},
\end{cases}
\end{equation*}
where $\alpha,\beta,\varepsilon_i,\delta$ are all constants satisfying some mild conditions; $u(t)$ is the control; $\tau_i$ is a random variable representing the default time; $\mathbb{I}_{\tau_i\leq t}$ is the indicator process; $\delta\overline{X}(t)$ is the mean-field term. Moreover, the performance function is given by
$$
J(u)=\mathbb{E}\left[\int_0^T\theta_1\ln (X(t)u(t))dt+\theta_2\ln X(T)\right],
$$
where $\theta_1$ and $\theta_2$ are positive constants. For given admissible control set $\mathcal{U}$, we would like to find $\widetilde{u}(\cdot)\in \mathcal{U}$ such that
$J(\widetilde{u})=\max\limits_{u(\cdot)\in \mathcal{U}}J(u)$.
\end{prob}
Problem  \ref{+8} is nothing but a finite horizon optimal control problem governed by the MMFSDE, of which the discipline is still not fully explored and much is desired to be done. Thus the first purpose of current paper is  to investigate the finite horizon optimal control problems governed by the systems of MMFSDEs under some mild conditions, and ensure the solvability of such systems.

After  achieving the  solvability of SDEs, one important issue of studying the well-posedness of SDEs is the stability of the solutions, in particular for the qualitative study and/or for the long time asymptotic behaviour of the solutions  \cite{shen2022stability}. In general, a solution to an SDE  is
stable if it is insensitive for small changes of the initial value or the parameters of the SDE. Up to now, the stability of the solutions to SDEs has been very well
developed since the seminal work \cite{khasminskii2012stochastic}, which extended  the celebrating concept of stability of deterministic dynamic systems introduced by Lyapunov \cite{liapounoff1907probleme}. Various stability,  were carried out in a series of works by Mao \cite{Mao1994Exponential, mao1999stability, MAO2008STOCHASTIC}, such as stochastic stability, stochastic asymptotical stability, moment exponential stability (including mean square exponential stability), almost sure  exponential stability, mean square polynomial stability etc. There are also fruitful results for the stability of the solutions to various type SDEs, such as SDEs with delays \cite{komori2020stability, luo2022some, singh2021asymptotic}, fractional SDEs \cite{luo2022some, omaba2021growth, shen2022stability, singh2021asymptotic} and stochastic functional differential equations \cite{cao2021razumikhin, shen2020controllability}. Moreover, the stability of the solutions to controlled SDEs (respectively MFSDEs) under optimal feedback controls has been established in \cite{lei2020feedback, li2020stabilisation, zhao2021stabilization} (respectively \cite{sun2024turnpike, wang2024optimal}). Very recently, the stability of the solutions to controlled MFSDEs under optimal feedback controls was established in \cite{sun2024turnpike, wang2024optimal}. However, to the best of our knowledge, there are no papers investigating the stability of the solutions to MMFSDEs. Thus, the second purpose of the current paper is to investigate the stability of the solutions to the controlled MMFSDEs under optimal feedback controls.

The main contributions of this paper can be summarized as follows: (i) The MMFSDE is considered in the sense of conditional expectation, which is quite different from the classical MFSDEs \cite{agram2018mean, wei2016stochastic}, and the optimal control problem governed by the MMFSDE is proposed; (ii) Backward induction method is derived for dividing our optimal control problems into several subproblems, while it is not available to solve traditional mean-field type (not conditional type) optimal control problems. (iii) Both sufficient and necessary maximum principles are obtained for such subproblems with random coefficients; (iv)  The existence and uniqueness results of solutions are obtained for both MMFSDEs and mean-field backward stochastic differential equation with multiple defaults (MMFBSDEs) under some mild conditions; (v) Mean square exponential stability and almost sure exponential stability are  established for the solutions to controlled MMFSDEs under  optimal feedback controls.

The rest of this paper is structured as follows. The next section introduces some necessary preliminaries including the enlargement progressive of filtration. After that in Section 3, the backward induction method is given, which transforms the global optimal control problem into several optimal  control subproblems, and both sufficient and necessary maximum principles are derived for these subproblems. In Section 4, the existence and uniqueness results of solutions are given for both MMFSDEs and MMFBSDEs, and the obtained results are applied to solve an optimal investment problem. Before we conclude, the mean square exponential stability and the  almost sure exponential stability are established for controlled MMFSDEs under optimal feedback control  in Section 5.

\section{Problem Formulation}
In this paper, we consider the following system of controlled MMFSDE in the complete probability space $(\Omega,\mathfrak{F},\mathfrak{F}_t,\mathbb{P})$:
\begin{equation}\label{SDE1}
\begin{cases}
dX(t)=b(t,X(t),M(X(t)),u(t),N(u(t)))dt+\sigma(t,X(t),M(X(t)),u(t),N(u(t)))dB_t\\
\qquad\qquad+h(t,X(t),M(X(t)),u(t),N(u(t)))d\mathbb{H}_t, \quad t\in(0,T],\\
X(0)=x_0,
\end{cases}
\end{equation}
where $X(\cdot)$ is the state variable, $u(\cdot)$ denotes the control variable, and $M(X(\cdot))$, $N(u(\cdot))$ are the conditional mean-field terms. Denote $\mathcal{F}=(\mathcal{F}_t)_{t\geq0}$ the right continuous, increasing and complete filtration generated by the one-dimensional Brownian motion $(B_t)_{t\geq0}$. Set $\mathcal{I}_0=\{0,1,\cdots,n\}$ and $\mathcal{I}=\{1,2,\cdots,n\}$. The default term
$$
\mathbb{H}_t=\mathbb{H}(t)=\sum \limits_{k=1}^{n}\mathbb{H}^k(t)=\sum \limits_{k=1}^{n}\mathbb{I}_{\tau_k\leq t}
$$
is generated by a sequence of ordered default times $\{\tau_k\}_{k\in \mathcal{I}}$ with
$$
0=\tau_0\leq \tau_1\leq \tau_2\leq\cdots \leq\tau_n\leq \tau_{n+1}=T \quad a.s..
$$
For $k\in\mathcal{I}$ and all $t\geq0$, we denote $\mathcal{N}^k=\sigma\left(\mathbb{I}_{\tau_k\leq u},u\in[0,t]\right)$. The global information is then naturally defined as the progressive enlargement of filtration  $\mathfrak{F}=\mathcal{F}\vee \mathcal{N}^1\vee\mathcal{N}^2\vee\cdots\vee\mathcal{N}^n$. The filtration $\mathfrak{F}=(\mathfrak{F}_t)_{t\geq0}$ is the smallest right continuous extension of $\mathcal{F}$ such that each $\tau_k\,(k\in\mathcal{I})$ becomes an $\mathfrak{F}$-stopping time. $x_0$ is an $\mathfrak{F}_{0}$-measurable and square integrable random variable, the coefficients $b$, $\sigma$ and $h$ will be introduced later. Furthermore, we give the following necessary assumptions.

\begin{assumption}\label{AH}
\begin{enumerate}[($\romannumeral1$)]
\item (Hypothesis $(\mathcal{H})$) Every c\`{a}dl\`{a}g $\mathcal{F}$-martingale remains an $\mathfrak{F}$-martingale.

\item
Suppose that there exists an $\mathfrak{F}$-predictable (respectively,  $\mathcal{F}$-predictable) process $\gamma^{\mathfrak{F}_k}$ (respectively, $\gamma_k=\gamma^{\mathcal{F}_k}$) with $\gamma^{\mathfrak{F}_k}(s)=\mathbb{I}_{s<\tau_k}\gamma^{\mathcal{F}_k}(s)$ such that
$$
A^k(t)=\mathbb{H}^k(t)-\int_0^t\gamma^{\mathfrak{F}^k}(s)ds=\mathbb{H}^k(t)-\int_0^{t\wedge \tau_k}\gamma^{\mathcal{F}_k}(s)ds\quad (t\geq0)
$$
is an $\mathfrak{F}$-martingale with jump time $\tau_k$. The process $\gamma^{\mathfrak{F}_k}$ (respectively, $\gamma^{\mathcal{F}}$) is called the $\mathfrak{F}$-intensity (respectively, $\mathcal{F}$-intensity) of $\tau_k$. In addition, assume that $\gamma^{\mathfrak{F}_k}$ is upper bounded.
\end{enumerate}
\end{assumption}

\begin{remark}
Hypothesis $(\mathcal{H})$ in Assumption \ref{AH} is also called the immersion property, i.e, the filtration $\mathcal{F}$ is immersed in $\mathfrak{F}$. Under Assumption \ref{AH}, for any $\mathfrak{F}$-optional process $X(t)$ with $\mathbb{E}[\int_0^T X(t)^2dt]<\infty$, the stochastic integral $\int_0^T X(t)dB_t$ is well defined \cite{Kharroubi2014Progressive}.
\end{remark}
\begin{remark}
In the rest of the paper we set $\Upsilon^{\mathfrak{F}}=\sum \limits_{k=1}^{n}\gamma^{\mathfrak{F}_k}$ and $\mathbb{A}(t)=\mathbb{H}(t)-\int_0^t\Upsilon^{\mathfrak{F}}ds$. It is easy to see that
\begin{align*}
\mathbb{E}[\mathbb{A}(t)|\mathfrak{F}_t]&=\mathbb{E}\left[\sum\limits_{k=1}^{n} \left(\mathbb{H}^k(t)-\int_0^t\gamma^{\mathfrak{F}_k}(s)ds\right)\Big|\mathfrak{F}_t\right]
=\mathbb{E}\left[\sum\limits_{k=1}^{n} A^k(t)\Big|\mathfrak{F}_t\right]\\
&=\sum\limits_{k=1}^{n} \mathbb{E}\left[A^k(t)|\mathfrak{F}_t\right]=\sum\limits_{k=1}^{n} A^k(t)=\mathbb{A}(t).
\end{align*}
Therefore, $\mathbb{A}(t)$ is an $\mathfrak{F}$-martingale.
\end{remark}

Next, we recall some useful sets and spaces, and the definition of the Wasserstein metric.
\begin{defn}
\begin{itemize}
\item[(i)] Let $\mathcal{O}(\mathfrak{F})$ (respectively, $\mathcal{P}(\mathfrak{F})$) denote the $\sigma$-algebra of $\mathfrak{F}$-optional (respectively, $\mathfrak{F}$-predictable) measurable subsets of $[0,T]\times \Omega$, i.e.,  $\mathcal{O}(\mathfrak{F})$ (respectively, $\mathcal{P}(\mathfrak{F})$) is  the $\sigma$-algebra generated by the right continuous (respectively, left continuous) $\mathfrak{F}$-adapted processes.

\item[(ii)] For a given Borel set $\mathbb{U}\subset\mathbb{R}$, let $\mathcal{O}(\mathfrak{F},\mathbb{U})$ denote the set of elements in $\mathcal{O}(\mathfrak{F})$ valued in $\mathbb{U}$.

\item[(iii)] $L^2([0,T])$ is the space of all square integrable functions $f$ with
$$
\|f\|_{L^2([0,T])}^2=\int_0^T|f(t)|^2dt<+\infty.
$$

\item[(iv)]For $T_0\in[0,T]$, $L^2_{T_0}(\mathbb{P})$ is the space of $\mathfrak{F}_{T_0}$-measurable random variable $\xi$ such that
$$
\|\xi\|_{L^2_{T_0}(\mathbb{P})}^2=\mathbb{E}[\xi^2]<+\infty.
$$

\item[(v)] $L_{T}^{\mathfrak{F}}$ is the set of all $\mathcal{O}(\mathfrak{F})$-measurable processes $X(t)$ such that
$$
\|X(\cdot)\|^2_{L_{T}^{\mathfrak{F}}}:=\mathbb{E}\left(\int_{0}^{T}|X(t)|^2dt\right)<+\infty,
$$
\item[(v)] $L_{+\infty}^{\mathfrak{F},\delta}$ is the set of all $\mathcal{O}(\mathfrak{F})$-measurable processes $X(t)$ such that
$$
\|X(\cdot)\|^2_{L_{+\infty}^{\mathfrak{F},\delta}}:=\mathbb{E}\left(\int_{0}^{+\infty}e^{-\delta t}|X(t)|^2dt\right)<+\infty,
$$
where $\delta$ is a given constant, and $H_{T}^{\mathfrak{F}}$ is the set of all $\mathcal{P}(\mathfrak{F})$-measurable processes $X(t)$ such that
$$
\|X(\cdot)\|^2_{H_{T}^{\mathfrak{F}}}:=\mathbb{E}\left(\int_{0}^{T}|X(t)|^2dt\right)<+\infty,
$$

\item[(vi)] $L_{T}^{\mathfrak{F},A^k}$ is the set of all $\mathcal{O}(\mathfrak{F})$-measurable processes $X(t)$ such that
$$
\|X(\cdot)\|^2_{L_{T}^{\mathfrak{F},A^k}}:=\mathbb{E}\left(\int_{0}^{T}\gamma^{\mathfrak{F}^{k}}(t)|X(t)|^2dt\right)<+\infty
$$
and $H_{T}^{\mathfrak{F},A^k}$ is the set of all $\mathcal{P}(\mathfrak{F})$-measurable processes $X(t)$ such that
$$
\|X(\cdot)\|^2_{H_{T}^{\mathfrak{F},A^k}}:=\mathbb{E}\left(\int_{0}^{T}\gamma^{\mathfrak{F}^{k}}(t)|X(t)|^2dt\right)<+\infty.
$$

\item[(vii)] $\mathcal{M}^k$, $\mathcal{G}^k$, $\mathcal{S}^k$ are the subfiltrations of $\mathfrak{F}$ such that $\mathcal{M}^k_t,\mathcal{G}^k_t,\mathcal{S}^k_t\subseteq \mathfrak{F}_t$ for all $t\in[0,T]$ and $\mathcal{M}^k_{\tau_k}=\mathcal{S}^k_{\tau_k}= \mathfrak{F}_{\tau_k}$.
\end{itemize}
\end{defn}

\begin{defn}
 Consider the space of probability measures $m$'s on $\mathbb{R}$ with $\int_{\mathbb{R}}x^2dm(x)<+\infty$. For any such $m$ and $m'$,  the 2-Wasserstein metric of them is defined by
$$
W_2(m,m')=\inf\limits_{\pi\in \Pi_2(m,m')}\left[\int_{\mathbb{R}\times\mathbb{R}}(x-x')^2\pi(dx,dx')\right]^{\frac{1}{2}},
$$
where $\Pi_2(m,m')$ represents the set of joint probability measures with respective marginals $m$ and $m'$. Denote $\mathcal{P}_2(\mathbb{R})$ the space equipped with the 2-Wasserstein metric.  The infimum is attainable such that there are random variables $X_{m}$ and $X_{m'}$ associated with $m$ and $m'$ respectively so that
$$
W_2(m,m')=\sqrt{\mathbb{E}\left[(X_{m}-X_{m'})^2\right]}.
$$
\end{defn}

Usually, we denote $\mathcal{S}^k_t$ the observed information up to time $t$, and use $\mathcal{M}^k_t$ and $\mathcal{G}^k_t$ to represent the information available to the state and the controller at time $t$, respectively. For instance, for $t\in[\tau_k,\tau_{k+1}]$, $\mathcal{S}^k_t=\mathfrak{F}_{\tau_k+(t-\tau_k-\delta_0)^{+}}$ represents the delayed information flow compared to $\mathfrak{F}_t$, where $\delta_0$ is a positive constant. And the mean-field term $N(u(t))$ can be removed when $\mathcal{G}^k_t=\mathcal{S}^k_t$. Moreover, for $X(\cdot),u(\cdot)\in L_{T}^{\mathfrak{F}}$, the conditional mean-field terms $M(X(\cdot))$ and $N(u(\cdot))$ are given as follows:
\begin{align*}
M(X(t))&=\sum\limits_{k=0}^{n-1}\mathbb{E}\left[X^k(t)\Big|\mathcal{M}^k_t\right]\mathbb{I}_{t\in [\tau_k,\tau_{k+1})}+\mathbb{E}\left[X^n(t)\Big|\mathcal{M}^n_t\right]\mathbb{I}_{t\in [\tau_n,T]},\\
N(u(t))&=\sum\limits_{k=0}^{n-1}\mathbb{E}\left[u^k(t)\Big|\mathcal{G}^k_t\right]\mathbb{I}_{t\in [\tau_k,\tau_{k+1})}+\mathbb{E}\left[u^n(t)\Big|\mathcal{G}^n_t\right]\mathbb{I}_{t\in [\tau_n,T]},
\end{align*}
where $X^i(t)$ and $u^i(t)$ ($i\in \mathcal{I}_0$) are determined by
\begin{align*}
X(t)&=\sum\limits_{k=0}^{n-1}X^k(t)\mathbb{I}_{t\in [\tau_k,\tau_{k+1})}+X^n(t)\mathbb{I}_{t\in [\tau_n,T]},\\
u(t)&=\sum\limits_{k=0}^{n-1}u^k(t)\mathbb{I}_{t\in [\tau_k,\tau_{k+1})}+u^n(t)\mathbb{I}_{t\in [\tau_n,T]}.
\end{align*}
The above decompositions can be ensured by Lemma 2.1 in \cite{Pham2010Stochastic}. Without ambiguity, we write $A(s)=A(s-)$ and $H(s)=H(s-)$ since there are only finite default jumps. In the sequel we regard the conditional mean-field terms as some elements in the metric space $\mathcal{P}_2(\mathbb{R})$, which is complete and separable. And so under mild conditions, the functionals of the conditional mean-field terms enjoys the Frechet derivatives defined as follows.

\begin{defn}
Let $\mathcal{X}$ and $\mathcal{Y}$ be two Banach spaces. A mapping $F: \mathcal{X}\rightarrow\mathcal{Y}$ is said to have a directional derivative (or G\^{a}teaux derivative) at $v\in \mathcal{X}$ in the direction $\varsigma\in \mathcal{X}$ if
$$
D_{\varsigma}F(v)=\lim \limits_{\varepsilon\rightarrow 0}\frac{1}{\varepsilon}(F(v+\varepsilon\varsigma)-F(v))
$$
exists in $\mathcal{Y}$. Moreover, $F$ is Fr\'{e}chet  differentiable at $v\in \mathcal{X}$ if there exists a continuous linear map $A:\mathcal{X}\rightarrow\mathcal{Y}$ such that
$$
\lim \limits_{\|\varsigma\|_{\mathcal{X}}\rightarrow 0}\frac{1}{\|\varsigma\|_{\mathcal{X}}}\|F(v+\varsigma)-F(v)-A(\varsigma)\|_{\mathcal{Y}}=0,
$$
where $A(h) =\langle A,h\rangle$ is the action of the liner operator $A$ on $h$. In this case $A$ is called the gradient (or Fr\'{e}chet derivative) of $F$ at $v$ which can be written as $A=\nabla_{v}F$. If $F$ is Fr\'{e}chet differentiable at $v$ with Fr\'{e}chet derivative $\nabla_{v}F$, then $F$ has a directional derivative in all directions $\varsigma\in \mathcal{X}$ and $D_{\varsigma}F(v)=\nabla_{v}F(\varsigma)=\langle\nabla_{v}F,\varsigma\rangle$.
\end{defn}

\begin{remark}
Clearly, if $F$ is a linear operator, then $\nabla_{v}F=F$ for all $v\in \mathcal{X}$. By choosing $\mathcal{X}=\mathcal{P}_2(\mathbb{R})\times\mathbb{R}$ and $\mathcal{Y}=\mathfrak{L}_2(\mathbb{R})$, the Fr\'{e}chet derivatives of mean-field terms can be well defined \cite{Bensoussan2023mean}.
\end{remark}

The admissible control set $\mathcal{U}^{ad}$ of the control variable is defined as follows:
$$
\mathcal{U}^{ad}=\left\{u\Big|u(t)=\sum\limits_{k=0}^{n-1}u^k(t)\mathbb{I}_{t\in [\tau_k,\tau_{k+1})}+u^n(t)\mathbb{I}_{t\in [\tau_n,\tau_{n+1}]},u^k\in\mathcal{U}_k^{ad},\;k\in \mathcal{I}_0\right\}.
$$
Here $\mathcal{U}_k^{ad}\,(k\in\mathcal{I}_0)$ is the separable metric subspace of $\mathcal{O}(\mathcal{S}^k,\mathbb{U}^k)$, where $(\mathbb{U}^k)_{k\in\mathcal{I}_0}$ is a given sequence of subsets in $\mathbb{R}$. We note here that the admissible control set is allowed to be variant. There exists an important distinction between the controlled system with the Poisson jumps and the one with defaults, i.e, the classical formulation  has to fix an admissible control set which is invariant during the time horizon, while the more general formulation can deal with  different admissible control sets between two default times. For these relevant works, we refer to \cite{Aksamit2017Enlargement, Pham2010Stochastic}. Furthermore, we study the optimal control problem in the sense of partial information, which is inspirited by some
interesting financial phenomena; for instance in some situations of real markets like insider trading, one investor may get more information than the others, and then, the investor can make a better decision than the others \cite{wang2017a, Dumitrescu2018Stochastic}.
For simplicity, we use the notation ${u}=({u}^0,{u}^1,\cdots,{u}^n)$ for $u\in \mathcal{U}^{ad}$ in the sequel. Without ambiguity, we set $u^k(t)=0$ for $t\notin[\tau_k,\tau_{k+1})$.

The following lemma will be frequently used in the sequel.
\begin{lemma}\label{ito}\cite{Peng2009BSDE}
Let $X_i(t)\;(i=1,2)$ be an It\^{o}'s jump-diffusion process given by
$$
dX_i(t)=b_i(t)dt+\sigma_i(t)dB_t+\sum \limits_{k=1}^{n}h_k^i(t)dA_t^k,
$$
where $b_i(\cdot),\sigma_i(\cdot)\in L_{T}^{\mathfrak{F}}$ and $h^i_k(\cdot)\in L_{T}^{\mathfrak{F},A^k}\;(k\in \mathcal{I})$. Then
\begin{align*}
d\left(X_1(t)X_2(t)\right)=&X_1(t)dX_2(t)+X_2(t)dX_1(t)+[dX_1(t),dX_2(t)]\\
=&X_1(t)dX_2(t)+X_2(t)dX_1(t)+\sigma_1(t)\sigma_2(t)dt+\sum \limits_{k=1}^{n}h^1_k(t)h^2_k(t)d\mathbb{H}^k_t.
\end{align*}
Especially, if $h^i_k(t)=h^i(t)$ for all $k\in \mathcal{I}$, then
$$
d\left(X_1(t)X_2(t)\right)=X_1(t)dX_2(t)+X_2(t)dX_1(t)+\sigma_1(t)\sigma_2(t)dt+h^1(t)h^2(t)d\mathbb{H}_t.
$$
\end{lemma}

The finite horizon optimal control problem considered in this paper is specified as follows.
\begin{prob}\label{pb}
The performance functional associated to the control $u=(u^0,u^1,\cdots,u^n)\in \mathcal{U}^{ad}=\mathcal{U}_0^{ad}\times\mathcal{U}_1^{ad}\times\cdots\times\mathcal{U}_n^{ad}$ takes the following form
$$
J(x_0,\mathfrak{L}(x_0),u)=\mathbb{E}\left[\int_0^Tf(t,X(t),M(X(t)),u(t),N(u(t)))dt+g(X(T),M(X(T)))\Big|\mathfrak{F}_{0}\right].
$$
Here, $X(t)$ is described by \eqref{SDE1}, $\mathfrak{L}(\cdot)$ represents the law of a random variable, and the running gain function $f:[0,T]\times \Omega\times \mathbb{R}\times\mathcal{P}_2(\mathbb{R})\times \mathbb{R}\times\mathcal{P}_2(\mathbb{R})\rightarrow\mathbb{R}$ and the terminal gain function $g:\Omega\times\mathbb{R}\times\mathcal{P}_2(\mathbb{R})\rightarrow\mathbb{R}$ are progressively measurable functions satisfying
\begin{equation}\label{++1}
|f(t,X,M,u,N)|+|g(X,M)|\leq C(1+X^2+| M|_{\mathcal{P}_2(\mathbb{R})}^2+u^2+|N|_{\mathcal{P}_2(\mathbb{R})}^2)
\end{equation}
for any $(t,X,M,u,N)\in[0,T]\times \mathbb{R}\times\mathcal{P}_2(\mathbb{R})\times \mathbb{R}\times\mathcal{P}_2(\mathbb{R})$ and for some constant $C>0$, a.s.. Moreover,  $f$ can be decomposed as
\begin{align*}
f(t,X(t),M(X(t)),u(t),N(u(t)))&=\sum \limits_{k=0}^{n-1}f^k(t,X^k(t),M(X^k(t)),u^k(t),N(u^k(t)))\mathbb{I}_{t\in [\tau_k,\tau_{k+1})}\\
&\quad\mbox{}+f^n(t,X^n(t),M(X^n(t)),u^n(t),N(u^n(t)))\mathbb{I}_{t\in [\tau_n,T]},
\end{align*}
where $f^k\,(k\in\mathcal{I}_0)$. The problem is to find the optimal control
$$\widehat{u}(t)=\sum\limits_{k=0}^{n-1}\widehat{u}^k(t)\mathbb{I}_{t\in [\tau_k,\tau_{k+1})}+\widehat{u}^n(t)\mathbb{I}_{t\in [\tau_n,T]} \in \mathcal{U}^{ad}$$
such that
\begin{equation*}
V(x_0,\mathfrak{L}(x_0))=J(x_0,\mathfrak{L}(x_0),\widehat{u})=\mathop{\esssup} \limits_{u\in\mathcal{U}^{ad}}J(x_0,\mathfrak{L}(x_0),u),
\end{equation*}¡¢
where the optimal controlled pair $(\widehat{X},\widehat{u})$ is governed by \eqref{SDE1}.
\end{prob}

\begin{remark}
In the running gain, there is a change of regimes after each default time, which is in the spirit of the concept of forward or progressive utility functions introduced in \cite{Musiela2009Portfolio}.
\end{remark}

\section{Backward Induction Method and Maximum Princeples}
In this section, we first give the backward induction method for solving Problem \ref{pb}. We transform the global optimal control problem into several subproblems. Then, we give both the sufficient and necessary conditions for these optimal control subproblems.

\subsection{Backward Induction Method}
To begin with, we make the following assumption throughout this paper.
\begin{assumption}
Suppose that $b,\sigma,h:[0,T]\times \Omega\times \mathbb{R}\times\mathcal{P}_2(\mathbb{R})\times \mathbb{R}\times\mathcal{P}_2(\mathbb{R})\rightarrow\mathbb{R}$ are all progressively measurable functions. Furthermore, assume that for any $\pi\in\{b,\sigma,h\}$, $\pi$ admits the following composition
$$
\pi^k(t,X^{k}(t),M(X^{k}(t)),u^{k}(t),N(u^{k}(t)))=\pi(t,X^{k}(t),M(X^{k}(t)),u^{k}(t),N(u^{k}(t)))\mathbb{I}_{t\in (\tau_k,\tau_{k+1}]}
$$
\end{assumption}
Under the above assumption, we are able to restrict equation \eqref{SDE1} on the random time interval $t\in[\tau_k,\tau_{k+1}]$ $(k\in\mathcal{I}_0)$:
\begin{equation*}
\begin{cases}
dX^{k}(t)=b^{k}(t,X^{k}(t),M(X^{k}(t)),u^{k}(t),N(u^{k}(t)))dt+\sigma^{k}(t,X^{k}(t),M(X^{k}(t)),u^{k}(t),N(u^{k}(t)))dB_t\\
\qquad\qquad\; +h^{k}(t,X^{k}(t),M(X^{k}(t)),u^{k}(t),N(u^{k}(t)))d\mathbb{H}^{k+1}_t, \quad t\in(\tau_k,\tau_{k+1}],\\
X^{k}(\tau_k)=X^{k-1}(\tau_k),
\end{cases}
\end{equation*}
which is equivalent to
\begin{equation}\label{SDE2}
\begin{cases}
dX^{k}(t)=b_{\gamma}^{k}(t,X^{k}(t),M(X^{k}(t)),u^{k}(t),N(u^{k}(t)))dt+\sigma^{k}(t,X^{k}(t),M(X^{k}(t)),u^{k}(t),N(u^{k}(t)))dB_t\\
\qquad\qquad\; +h^{k}(t,X^{k}(t),M(X^{k}(t)),u^{k}(t),N(u^{k}(t)))dA^{k+1}_t, \quad t\in(\tau_k,\tau_{k+1}],\\
X^{k}(\tau_k)=X^{k-1}(\tau_k).
\end{cases}
\end{equation}
Here
\begin{align*}
&\quad\; b^k_{\gamma}(t,X^{k}(t),M(X^{k}(t)),u^{k}(t),N(u^{k}(t)))\\
&=b^k(t,X^{k}(t),M(X^{k}(t)),u^{k}(t),N(u^{k}(t)))+\gamma^{\mathfrak{F}^{k+1}}(t)h(t,X^{k}(t),M(X^{k}(t)),u^{k}(t),N(u^{k}(t))),
\end{align*}
$X^0(0)=X^{-1}(0)=x_0\in \mathbb{R}$, $X^n(T)=X(T)$ and $h^n\equiv 0$.
We show in Subsection 4.1 that under some normal conditions, there exists a unique square integrable optional solution to \eqref{SDE1}. Since Lemma 2.1 in \cite{Pham2010Stochastic} (or Theorem 6.5 of \cite{Song2014optimal}) provides a key decomposition of optional processes with respect to the progressive enlargement of filtration, we can see that the solution $X(t)$ of \eqref{SDE1} has the following decomposition:
\begin{equation}\label{112}
X(t)=\sum\limits_{k=0}^{n-1}X^k(t)\mathbb{I}_{t\in [\tau_k,\tau_{k+1})}+X^n(t)\mathbb{I}_{t\in [\tau_n,\tau_{n+1}]},
\end{equation}
where $X^k(t)$ is governed by \eqref{SDE2}.

 Now we use the decomposition to derive the following backward induction method for the performance function.
\begin{prop}\label{prop perfun}
Consider the following backward induction sequence ($k\in\mathcal{I}_0$)
\begin{align*}
J^n(X^{n-1}(\tau_n),\mathfrak{L}^{n-1}(X^{n-1}(\tau_n)),u^n)&=\mathbb{E}\Big[\int_{\tau_n}^Tf^{n}(t,X^n(t),M(X^n(t)),u^n(t),N(u^n(t)))dt\\
&\qquad\mbox{}+\mathfrak{L}^{n}(X^n(T),M(X^n(T)))\Big|\mathfrak{F}_{\tau_n}\Big],\\
J^{k}(X^{k-1}(\tau_k),\mathfrak{L}^{k-1}(X^{k-1}(\tau_k)),u^k,\cdots,u^n)&=\mathbb{E}\Big[\int_{\tau_k}^{\tau_{k+1}}f^{k}(t,X^k(t),M(X^k(t)),u^k(t),N(u^k(t)))dt\\
&\qquad\mbox{}+J^{k+1}(X^k(\tau_{k+1}),\mathfrak{L}^{k}(X^{k}(\tau_{k+1})),u^{k+1},\cdots,u^n)\Big|\mathfrak{F}_{\tau_k}\Big]
\end{align*}
with $\mathfrak{L}^{n}=g$ and $\mathfrak{L}^{-1}=\mathfrak{L}$. Then $J^{0}(x_0,\mathfrak{L}(x_0),u)=J^{0}(X^{-1}(0),\mathfrak{L}^{-1}(X^{-1}(0)),u^0,\cdots,u^n)=J(x_0,\mathfrak{L}(x_0),u)$.
\end{prop}
\begin{proof}
Since $\mathbb{E}\left[X^{k-1}(\tau_k)\Big|\mathcal{M}_{\tau_k}\right]=\mathbb{E}\left[X^{k-1}(\tau_k)\Big|\mathfrak{F}_{\tau_k}\right]=X^{k-1}(\tau_k)$, it follows that $J^{k}(X^{k-1}(\tau_k),u^k,\cdots,u^n)$ and $J^n(X^{n-1}(\tau_n),u^n)$ are well-defined. Clearly,
\begin{align*}
&\quad\;J^{n-1}(X^{n-2}(\tau_{n-1}),\mathfrak{L}^{n-2}(X^{n-2}(\tau_{n-1})),u^{n-1},u^n)\\
&=\mathbb{E}\left[\int_{\tau_{n-1}}^{\tau_{n}}f^{n-1}(t,X^{n-1}(t),M(X^{n-1}(t)),u^{n-1}(t),N(u^{n-1}(t)))dt+J^{n}(X^{n-1}(\tau_n),u^n)\Big|\mathfrak{F}_{\tau_{n-1}}\right]\\
&=\mathbb{E}\left[\int_{\tau_{n-1}}^{\tau_{n}}f(t,X^{n-1}(t),M(X^{n-1}(t)),u^{n-1}(t),N(u^{n-1}(t)))dt\Big|\mathfrak{F}_{\tau_{n-1}}\right]\\
&\quad\mbox{}+\mathbb{E}\left[\mathbb{E}\left[\int_{\tau_n}^Tf(t,X^n(t),M(X^n(t)),u^n(t),N(u^n(t)))dt+g(X^n(T),M(X^n(T)))\Big|\mathfrak{F}_{\tau_{n}}\right]\Big|\mathfrak{F}_{\tau_{n-1}}\right]\\
&=\mathbb{E}\left[\int_{\tau_{n-1}}^{\tau_{n}}f^{n-1}(t,X^{n-1}(t),M(X^{n-1}(t)),u^{n-1}(t),N(u^{n-1}(t)))dt\Big|\mathfrak{F}_{\tau_{n-1}}\right]\\
&\quad\mbox{}+\mathbb{E}\left[\int_{\tau_n}^Tf^{n}(t,X^n(t),M(X^n(t)),u^n(t),N(u^n(t)))dt+g(X^n(T),M(X^n(T)))\Big|\mathfrak{F}_{\tau_{n-1}}\right]\\
&=\mathbb{E}\Big[\int_{\tau_{n-1}}^{T}f^{n-1}(t,X(t),M(X(t)),u(t),N(u(t)))\mathbb{I}_{t\in [\tau_{n-1},\tau_{n})}\\
&\quad\mbox{}+f^n(t,X(t),M(X(t)),u(t),N(u(t)))\mathbb{I}_{t\in [\tau_n,T]}dt+g(X(T),M(X(T)))\Big|\mathfrak{F}_{\tau_{n-1}}\Big].
\end{align*}
Noticing that $X^n(T)=X(T)$ and $\tau_{0}=0$, it follows from the backward induction that
\begin{align*}
J^{0}(x_0,\mathfrak{L}(x_0),u)&=J^{0}(X^{-1}(0),\mathfrak{L}^{-1}(X^{-1}(0)),u^0,\cdots,u^n)\\
&=\mathbb{E}\Big[\int_{\tau_0}^{T}\sum \limits_{i=0}^{n-1}f^i(t,X(t),M(X(t)),u(t),N(u(t)))\mathbb{I}_{t\in [\tau_i,\tau_{i+1})}\\
&\quad\mbox{}+f^n(t,X(t),M(X(t)),u(t),N(u(t)))\mathbb{I}_{t\in [\tau_n,T]}dt+g(X^n(T),M(X^n(T)))\Big|\mathfrak{F}_{\tau_{0}}\Big]\\
&=\mathbb{E}\Big[\int_{0}^{T}f(t,X(t),M(X(t)),u(t),N(u(t)))dt+g(X^n(T),M(X^n(T)))\Big|\mathfrak{F}_{0}\Big].
\end{align*}
This ends the proof.
\end{proof}

The following theorem provides a decomposition for the optimal control of Problem \ref{pb}, which is essentially a conditional mean-field type dynamic programming principle.
\begin{thm}\label{multi}
The optimal control $\widetilde{u}=(\widetilde{u}^0,\cdots,\widetilde{u}^n)$ can be obtained by solving each optimal control $u^k\,(k\in\mathcal{I}_0)$ of the following subproblems ($k\in \mathcal{I}_0$):
\begin{align}\label{1}
V^k(X^{k-1}(\tau_k),\mathfrak{L}^{k-1}(X^{k-1}(\tau_k)))&=J^{k}(X^{k-1}(\tau_k),\mathfrak{L}^{k-1}(X^{k-1}(\tau_k)),\widetilde{u}^k,\widetilde{u}^{k+1},\cdots,\widetilde{u}^n)\nonumber\\
&=\mathop{\esssup} \limits_{u^{k}\in \mathcal{U}_k^{ad}}J^{k}(X^{k-1}(\tau_k),\mathfrak{L}^{k-1}(X^{k-1}(\tau_k)),{u}^k,\widetilde{u}^{k+1},\cdots,\widetilde{u}^n).
\end{align}
\end{thm}

\begin{proof}
Omitting the elements $\widetilde{u}^{k+1},\cdots,\widetilde{u}^n$ $(k\in\mathcal{I}_0)$, it follows that
\begin{align}\label{+3}
V^k(X^{k-1}(\tau_k),\mathfrak{L}^{k-1}(X^{k-1}(\tau_k)))&=J^{k}(X^{k-1}(\tau_k),\mathfrak{L}^{k-1}(X^{k-1}(\tau_k)),\widetilde{u}^k)\nonumber\\
&=\mathop{\esssup} \limits_{u^{k}\in \mathcal{U}_k^{ad}}J^{k}(X^{k-1}(\tau_k),\mathfrak{L}^{k-1}(X^{k-1}(\tau_k)),u^{k})\nonumber\\
&=\mathop{\esssup} \limits_{u^{k}\in \mathcal{U}_k^{ad}}\mathbb{E}\Big[\int_{\tau_k}^{\tau_{k+1}}f^{k}(t,X^k(t),M(X^k(t)),u^k(t),N(u^k(t)))dt\nonumber\\
&\qquad\qquad\mbox{}+J^{k+1}(X^k(\tau_{k+1}),\mathfrak{L}^k(X^{k}(\tau_{k+1})),\widetilde{u}^{k+1},\cdots,\widetilde{u}^n)\Big|\mathfrak{F}_{\tau_k}\Big]\nonumber\\
&=\mathop{\esssup} \limits_{u^{k}\in \mathcal{U}_k^{ad}}\mathbb{E}\Big[\int_{\tau_k}^{\tau_{k+1}}f^{k}(t,X^k(t),M(X^k(t)),u^k(t),N(u^k(t)))dt\nonumber\\
&\qquad\qquad\mbox{}+V^{k+1}(X^k(\tau_{k+1}),\mathfrak{L}^k(X^{k}(\tau_{k+1})))\Big|\mathfrak{F}_{\tau_k}\Big],
\end{align}
where $V^{n+1}(X^n(\tau_{n+1}),\mathfrak{L}^n(X^{n}(\tau_{n+1})))=g(X(T),M(X(T)))$. Here $(X^k,u^k)$ is the solution to the $(k+1)$-th equation of \eqref{SDE2}, $X^k(\tau_{k+1})$ is the terminal value of this equation and $V^k(X^{k-1}(\tau_k),\mathfrak{L}^{k-1}(X^{k-1}(\tau_k)))$ is an $\mathfrak{F}_{\tau_k}$-measurable random variable.
Thus, we can set
\begin{align*}
\mathbb{V}^k(X^{k-1}(\tau_k),\mathfrak{L}^{k-1}(X^{k-1}(\tau_k)))&=J^{k}(X^{k-1}(\tau_k),\mathfrak{L}^{k-1}(X^{k-1}(\tau_k)),\widehat{u}^k,\widehat{u}^{k+1},\cdots,\widehat{u}^n)\\
&=\mathop{\esssup} \limits_{({u}^{k},\cdots,{u}^n)\in \mathcal{U}_k^{ad}\times\cdots\times\mathcal{U}_n^{ad}}J^{k}(X^{k-1}(\tau_k),\mathfrak{L}^{k-1}(X^{k-1}(\tau_k)),{u}^k,{u}^{k+1},\cdots,{u}^n).
\end{align*}
We now aim to prove that
\begin{equation}\label{2}
V^k(X^{k-1}(\tau_k),\mathfrak{L}^{k-1}(X^{k-1}(\tau_k)))=\mathbb{V}^k(X^{k-1}(\tau_k),\mathfrak{L}^{k-1}(X^{k-1}(\tau_k))), \quad k\in\mathcal{I}_0.
\end{equation}
Trivially, equation \eqref{2} holds for $k=n$. By Proposition \ref{prop perfun}, for any $ ({u}^{n-1},{u}^n) \in \mathcal{U}_{n-1}^{ad}\times\mathcal{U}_n^{ad} $ with $k\in\mathcal{I}$, we have
\begin{align*}
\mathbb{V}^{n-1}(X^{n-2}(\tau_{n-1}),\mathfrak{L}^{n-2}(X^{n-2}(\tau_{n-1}))&\geq J^{n-1}(X^{n-2}(\tau_{n-1}),\mathfrak{L}^{n-2}(X^{n-2}(\tau_{n-1}),{u}^{n-1},u^n)\\
&=\mathbb{E}\Big[\int_{\tau_{n-1}}^{\tau_{n}}f^{n-1}(t,X^{n-1}(t),M(X^{n-1}(t)),u^{n-1}(t),N(u^{n-1}(t)))dt\\
&\qquad\mbox{}+J^{n}(X^{n-1}(\tau_{n}),\mathfrak{L}^{n-1}(X^{n-1}(\tau_{n}),u^n))\Big|\mathfrak{F}_{\tau_{n-1}}\Big].
\end{align*}
Taking ${u}^n=\widetilde{u}^n$, one has
\begin{align}\label{3}
\mathbb{V}^{n-1}(X^{n-2}(\tau_{n-1}),\mathfrak{L}^{n-2}(X^{n-2}(\tau_{n-1}))
&\geq\mathbb{E}\Big[\int_{\tau_{n-1}}^{\tau_{n}}f^{n-1}(t,X^{n-1}(t),M(X^{n-1}(t)),u^{n-1}(t),N(u^{n-1}(t)))dt\nonumber\\
&\qquad\mbox{}+V^{n}(X^{n-1}(\tau_{n}),\mathfrak{L}^{n-1}(X^{n-1}(\tau_{n}))\Big|\mathfrak{F}_{\tau_{n-1}}\Big]\nonumber\\
&=V^{n-1}(X^{n-2}(\tau_{n-1}),\mathfrak{L}^{n-2}(X^{n-2}(\tau_{n-1})).
\end{align}
Recalling that the admissible control set is a separable metric space,  the measurable selection result (see,  for example,  \cite{Pham2018Bellman, Wagner1980Survey, yong2019Stochastic}) indicates that for arbitrary $\varepsilon>0$, there exist $\varepsilon$-optimal controls $u^{\varepsilon,n-1}\in \mathcal{U}_{n-1}^{ad}$ of $\mathbb{V}^{n-1}(X^{n-2}(\tau_{n-1}),\mathfrak{L}^{n-2}(X^{n-2}(\tau_{n-1}))$ and $u^{\varepsilon,n}\in \mathcal{U}_{n}^{ad}$ of $\mathbb{V}^{n}(X^{n-1}(\tau_{n}),\mathfrak{L}^{n-1}(X^{n-1}(\tau_{n}))$
such that
\begin{align*}
&\quad\; \mathbb{V}^{n-1}(X^{n-2}(\tau_{n-1}),\mathfrak{L}^{n-2}(X^{n-2}(\tau_{n-1}))-\varepsilon\\
&<\mathbb{E}\Big[\int_{\tau_{n-1}}^{\tau_{n}}f^{n-1}(t,X^{\varepsilon,n-1}(t),M(X^{\varepsilon,n-1}(t)),u^{\varepsilon,n-1}(t),N(u^{\varepsilon,n-1}(t)))dt\\
&\qquad\mbox{}+J^{n}(X^{\varepsilon,n-1}(\tau_{n}),\mathfrak{L}^{n-1}(X^{\varepsilon,n-1}(\tau_{n})),u^{\varepsilon,n})\Big|\mathfrak{F}_{\tau_{n-1}}\Big]\\
&\leq\mathbb{E}\Big[\int_{\tau_{n-1}}^{\tau_{n}}f^{n-1}(t,X^{\varepsilon,n-1}(t),M(X^{\varepsilon,n-1}(t)),u^{\varepsilon,n-1}(t),N(u^{\varepsilon,n-1}(t)))dt\\
&\qquad\mbox{}+V^{n}(X^{\varepsilon,n-1}(\tau_{n}),\mathfrak{L}^{n-1}(X^{\varepsilon,n-1}(\tau_{n})))\Big|\mathfrak{F}_{\tau_{n-1}}\Big]\\
&\leq V^{n-1}(X^{n-2}(\tau_{n-1}),\mathfrak{L}^{n-2}(X^{n-2}(\tau_{n-1}))) \quad\mbox{a.s.}.
\end{align*}
Taking $\varepsilon\rightarrow 0$, one has
\begin{equation}\label{4}
\mathbb{V}^{n-1}(X^{n-2}(\tau_{n-1}),\mathfrak{L}^{n-2}(X^{n-2}(\tau_{n-1})))\leq V^{n-1}(X^{n-2}(\tau_{n-1})),\mathfrak{L}^{n-2}(X^{n-2}(\tau_{n-1}))).
\end{equation}
Combining \eqref{3} and \eqref{4}, we can obtain
$$\mathbb{V}^{n-1}(X^{n-2}(\tau_{n-1}),\mathfrak{L}^{n-2}(X^{n-2}(\tau_{n-1})))= V^{n-1}(X^{n-2}(\tau_{n-1}),\mathfrak{L}^{n-2}(X^{n-2}(\tau_{n-1}))).$$
Then \eqref{2} is easily derived by using the backward induction.
Furthermore,
$$V^0(x_0,\mathfrak{L}(x_0))=\mathbb{V}^0(x_0,\mathfrak{L}(x_0))=V(x_0,\mathfrak{L}(x_0)), \quad \text{ $k=0$. }$$
This completes the proof.
\end{proof}

Theorem \ref{multi} shows that $(\widetilde{u}^0,\cdots,\widetilde{u}^n)$ obtained by \eqref{1} is the global optimal control.  Thus, solving  problem  \eqref{pb} is equivalent to solve the following subproblem:

\begin{subprob}\label{spb}
Find $\widetilde{u}^k\in \mathcal{U}_k^{ad}$ such that
$$
J^{k}(X^{k-1}(\tau_k),\mathfrak{L}^{k-1}(X^{k-1}(\tau_k)),\widetilde{u}^k)=\mathop{\esssup} \limits_{u^{k}\in \mathcal{U}_k^{ad}}J^{k}(X^{k-1}(\tau_k),\mathfrak{L}^{k-1}(X^{k-1}(\tau_k)),{u}^k), \quad k\in\mathcal{I}_0.
$$
\end{subprob}

\subsection{A sufficient maximum principle}

In this subsection, we derive the sufficient version of the maximum principle for Subproblem \ref{spb}.  The following lemma plays an important role in optimal controls problems of MFSDEs with random coefficients.

\begin{lemma}\label{lm}
For any $X(\cdot),Y(\cdot)\in L_{T}^{\mathfrak{F}}$ and $k\in \mathcal{I}_0$, we have $M(X(\cdot)),M(Y(\cdot))\in L_{T}^{\mathfrak{F}}$ and
$$
|M(X(t))-M(Y(t))|^2\leq \mathbb{E}\left[|X(t)-Y(t)|^2\Big|\mathcal{M}^k_{t}\right].
$$
Moreover, for $t\in[\tau_k,\tau_{k+1}]$ with $k\in \mathcal{I}_0$,
\begin{align*}
\mathbb{E}\left[M(X(t))Y(t)\Big|\mathfrak{F}_{\tau_k}\right]&=\mathbb{E}\left[X(t)M(Y(t))\Big|\mathfrak{F}_{\tau_k}\right],\\
\mathbb{E}\left[\gamma^{\mathfrak{F}^{k+1}}(t)M(X(t))Y(t)\Big|\mathfrak{F}_{\tau_k}\right]&=\mathbb{E}\left[X(t)M(\gamma^{\mathfrak{F}^{k+1}}(t)Y(t))\Big|\mathfrak{F}_{\tau_k}\right].
\end{align*}
\end{lemma}

\begin{proof}
The first inequality follows from Cauchy's inequality immediately. Noticing that $\mathcal{M}^k_{\tau_k}=\mathfrak{F}_{\tau_k}$, one has
\begin{align*}
\mathbb{E}\left[M(X(t))Y(t)\Big|\mathfrak{F}_{\tau_k}\right]&=\mathbb{E}\left[\mathbb{E}\left[\mathbb{E}\left[X(t)\Big|\mathcal{M}^k_{t}\right]\cdot Y(t)\Big|\mathcal{M}^k_{t}\right]\Big|\mathfrak{F}_{\tau_k}\right]\\
&=\mathbb{E}\left[\mathbb{E}\left[X(t)\Big|\mathcal{M}^k_{t}\right]\cdot\mathbb{E}\left[Y(t)\Big|\mathcal{M}^k_{t}\right]\Big|\mathfrak{F}_{\tau_k}\right]\\
&=\mathbb{E}\left[\mathbb{E}\left[\mathbb{E}\left[Y(t)\Big|\mathcal{M}^k_{t}\right]\cdot X(t)\Big|\mathcal{M}^k_{t}\right]\Big|\mathfrak{F}_{\tau_k}\right]\\
&=\mathbb{E}\left[X(t)M(Y(t))\Big|\mathfrak{F}_{\tau_k}\right].
\end{align*}
The second equality follows similarly.
\end{proof}

We also need to define following Hamiltonian functional:
\begin{align}\label{5}
H^k&:[0,T]\times \Omega\times\mathbb{R}\times\mathcal{P}_2(\mathbb{R})\times\mathbb{R}\times\mathcal{P}_2(\mathbb{R})\times\mathbb{R}\times\mathbb{R}\times\mathbb{R}\rightarrow \mathbb{R}\nonumber\\
H^k(t)&=H^k(t,x^k,M^k,u^k,N^k,p^k,q^k,r^k)\nonumber\\
&=f^k(t,x^k,M^k,u^k,N^k)+b^k_{\gamma}(t,x^k,M^k,u^k,N^k)p^k+\sigma^k(t,x^k,M^k,u^k,N^k)q^k\nonumber\\
&\quad\mbox{}+\gamma^{\mathfrak{F}^{k+1}}h(t,x^k,M^k,u^k,N^k)r^k.
\end{align}

In the sequel, we assume that the functions $b^k$, $\sigma^k$, $\gamma^k$ and $f^k$ all admit bounded Fr\'{e}chet derivatives with respect to $x^k$, $M^k$, $u^k$ and $N^k$, respectively. In addition, $V^{k+1}(X^k(\tau_{k+1}),\mathfrak{L}^{k}(X^k(\tau_{k+1})))$ admits bounded Fr\'{e}chet derivatives with respect to $X^k(\tau_{k+1})$ and $\mathfrak{L}^{k}(X^k(\tau_{k+1}))$.

Now we associate the adjoint BSDE to Hamiltonian \eqref{5} in the unknown triple $\left(p^k,q^k,r^k\right)$ ($k\in\mathcal{I}_0$) satisfying
\begin{equation*}
\begin{cases}
dp^k(t)=-\left(\frac{\partial H^k}{\partial x^k}(t)+M(\nabla_{M}H^k(t))\right)dt+q^k(t)dB_t+r^k(t)dA^{k+1}_t,\quad t\in[\tau_k,\tau_{k+1}],\\
p^k(\tau_{k+1})=\frac{\partial V^{k+1}}{\partial x^k}(X^k(\tau_{k+1}),\mathfrak{L}^{k}(X^k(\tau_{k+1})))+\nabla_{\mathfrak{L}^k}V^{k+1}(X^k(\tau_{k+1}),\mathfrak{L}^{k}(X^k(\tau_{k+1}))),
\end{cases}
\end{equation*}
where $M(\nabla_{M}H^k(t))=\mathbb{E}\left[\nabla_{M}H^k(t)\Big|\mathcal{M}^k_t\right]$.

\begin{example}
If the state process satisfies ($k\in \mathcal{I}_0$)
\begin{equation*}
\begin{cases}
dX^{k}(t)=u^k(t)dt+M(X^{k}(t))d\mathbb{H}^{k+1}_t, \quad t\in(\tau_k,\tau_{k+1}],\\
X^{k}(\tau_k)=X^{k-1}(\tau_k)
\end{cases}
\end{equation*}
and the performance functional takes the form
$$
J(X^k,u^k)=-\frac{1}{2}\mathbb{E}\left[\int_{\tau_{k}}^{\tau_{k+1}}|X^k(t)|^2+|u^k(t)|^2dt+|X^k(\tau_{k+1})|^2\right],
$$
then the Hamiltonian functional is
$$
H^k=-\frac{1}{2}\left[|X^k|^2+|u^k|^2\right]+(u^k+M(X^k)\gamma^{\mathfrak{F}^{k+1}})p+\gamma^{\mathfrak{F}^{k+1}}M(X^k)r^k
$$
and the adjoint BSDE is
\begin{equation*}
\begin{cases}
dp^k(t)=-\left(X^k(t)-\mathbb{E}\left[\gamma^{\mathfrak{F}^{k+1}}(t)p^k(t)\Big|\mathcal{M}^k_{t}\right]-\mathbb{E}\left[\gamma^{\mathfrak{F}^{k+1}}(t)r^k(t)\Big|\mathcal{M}^k_{t}\right]\right)dt\\
\qquad\qquad+r^k(t)dA^{k+1}_t,\quad t\in[\tau_k,\tau_{k+1}],\\
p^k(\tau_{k+1})=-X^k(\tau_{k+1}).
\end{cases}
\end{equation*}
\end{example}

\begin{thm}\label{sufficient}
Assume that
\begin{itemize}
\item[(i)] (Concavity) For $t\in [\tau_k,\tau_{k+1}]$ with $k\in \mathcal{I}_0$, the functions
\begin{align*}
(x^k,M^k,u^k,N^k)&\rightarrow H(t,x^k,M^k,u^k,N^k,\widehat{p}^k,\widehat{q}^k,\widehat{r}^k),\\
(X^k(\tau_{k+1}),\mathfrak{L}^{k}(X^k(\tau_{k+1})))&\rightarrow V^{k+1}(X^k(\tau_{k+1}),\mathfrak{L}^{k}(X^k(\tau_{k+1})))
\end{align*}
are concave a.s..
\item[(ii)] (Maximum condition) For $t\in [\tau_k,\tau_{k+1}]$ with $k\in \mathcal{I}_0$,
$$
\mathbb{E}\left[\widehat{H}(t,\widehat{x}^k,\widehat{M}^k,\widehat{u}^k,\widehat{N}^k,\widehat{p}^k,\widehat{q}^k,\widehat{r}^k)\Big|\mathcal{G}^k_t\right]=\mathop{\esssup}  \limits_{u^k\in \mathcal{U}_k^{ad}}\mathbb{E}\left[\widehat{H}(t,\widehat{x}^k,\widehat{M}^k,{u^k},N^k,\widehat{p}^k,\widehat{q}^k,\widehat{r}^k)\Big|\mathcal{S}^k_t\right].
$$
\end{itemize}
Then, $\widehat{u}^k$ is an optimal control of Subproblem \ref{spb}.
\end{thm}

\begin{proof}
Let
\begin{equation*}
J^k(X^{k-1}(\tau_k),L(X^{k-1}(\tau_k)),u^k)-J^k(\widehat{X}^{k-1}(\tau_k),L(\widehat{X}^{k-1}(\tau_k)),\widehat{u}^k)=I_1+I_2, \quad k\in \mathcal{I}_0,
\end{equation*}
where
\begin{align*}
I_1&=\mathbb{E}\bigg[\int_{\tau_k}^{\tau_{k+1}}f^k(t,X^k(t),M(X^k(t)),u^k(t),N(u^k(t)))-f^k(t,\widehat{X}^k(t),M(\widehat{X}^k(t)),\widehat{u}^k(t),\widehat{N}(u^k(t)))dt\Big|\mathfrak{F}_{\tau_k}\bigg],\\
I_2&=\mathbb{E}\left[V^{k+1}(X^k(\tau_{k+1}),\mathfrak{L}^{k}(X^k(\tau_{k+1})))-V^{k+1}(\widehat{X}^k(\tau_{k+1}),\mathfrak{L}^{k}(\widehat{X}^k(\tau_{k+1})))\Big|\mathfrak{F}_{\tau_k}\right].
\end{align*}
Setting $\widetilde{X}^k(\tau_{k+1})=X^k(\tau_{k+1}))-\widehat{X}^k(\tau_{k+1})$ and applying Lemma \ref{lm}, one has
$$
\mathbb{E}\left[\int_{\tau_k}^{\tau_{k+1}}\langle\nabla_{M}H^k(t),M(\widetilde{X}^k(t))\rangle dt\Big|\mathfrak{F}_{\tau_k}\right]=\mathbb{E}\left[\int_{\tau_k}^{\tau_{k+1}} M(\nabla_{M}H^k(t))\widetilde{X}^k(t)dt\Big|\mathfrak{F}_{\tau_k}\right].
$$
Thus, by It\^{o}'s formula and the martingale property of $\int_{\tau_k}^{\tau_{k+1}}\widehat{r}^k(t)\widetilde{h}^k(t)dA^{k+1}(t)$, we have
\begin{align*}
I_2&\leq\mathbb{E}\left[\{\frac{\partial V^{k+1}}{\partial x^k}(X^k(\tau_{k+1}),\mathfrak{L}^{k}(X^k(\tau_{k+1})))+\nabla_{\mathfrak{L}^k}V^{k+1}(X^k(\tau_{k+1}),\mathfrak{L}^{k}(X^k(\tau_{k+1})))\}\cdot\widetilde{X}^k(\tau_{k+1})\Big|\mathfrak{F}_{\tau_k}\right]\\
&=\mathbb{E}\left[\widehat{p}^k(\tau_{k+1})\widetilde{X}^k(\tau_{k+1})\Big|\mathfrak{F}_{\tau_k}\right]=\mathbb{E}\left[\widehat{p}^k(\tau_{k+1})\widetilde{X}^k(\tau_{k+1})-\widehat{p}^k(\tau_{k})\widetilde{X}(\tau_{k})\Big|\mathfrak{F}_{\tau_k}\right]\\
&=\mathbb{E}\left[\int_{\tau_k}^{\tau_{k+1}}\widehat{p}^k(t)d\widetilde{X}^k(t)+\widetilde{X}^k(t)d\widehat{p}^k(t)+\widehat{q}^k(t)\widetilde{\sigma}^k(t)dt+\widehat{r}^k(t)\widetilde{h}^k(t)d\mathbb{H}^{k+1}_t\Big|\mathfrak{F}_{\tau_k}\right]\\
&=\mathbb{E}\Big[\int_{\tau_k}^{\tau_{k+1}}\widehat{p}^k(t)\widetilde{b}^k_{\gamma}(t)dt-\widetilde{X}^k(t)\frac{\partial H^k}{\partial x^k}(t)-\langle\nabla_{M}H^k(t),M(\widetilde{X}^k(t))\rangle\\
&\quad\mbox{}+\widehat{q}^k(t)\widetilde{\sigma}^k(t)+\widehat{r}^k(t)\widetilde{h}^k(t)\gamma^{\mathfrak{F}_k}(t)dt\Big|\mathfrak{F}_{\tau_k}\Big].
\end{align*}
It follows from \eqref{5} that
$$
I_1=\mathbb{E}\left[\int_{\tau_k}^{\tau_{k+1}}\widetilde{H}^k(t)-\widehat{p}^k(t)\widetilde{b}^k_{\gamma}(t)-\widehat{q}^k(t)\widetilde{\sigma}^k(t)-\widehat{r}^k(t)\widetilde{h}^k(t)\gamma^{\mathfrak{F}^{k+1}}(t)dt\Big|\mathfrak{F}_{\tau_k}\right].
$$
Since $\mathcal{S}_{\tau_k}=\mathfrak{F}_{\tau_k}$,  the concave condition imposed on $H^k(t)$ leads to
\begin{align*}
I_1+I_2&= J^k(X^{k-1}(\tau_k),L(X^{k-1}(\tau_k)),u^k)-J^k(\widehat{X}^{k-1}(\tau_k),L(\widehat{X}^{k-1}(\tau_k)),\widehat{u}^k)\\
&\leq\mathbb{E}\left[\int_{\tau_k}^{\tau_{k+1}}\widetilde{H}^k(t)-\widetilde{X}^k(t)\frac{\partial H^k}{\partial x^k}(t)-\langle\nabla_{M}H^k(t),M(\widetilde{X}^k(t))\rangle dt\Big|\mathfrak{F}_{\tau_k}\right]\\
&\leq\mathbb{E}\left[\int_{\tau_k}^{\tau_{k+1}}\widetilde{u}^k(t)\frac{\partial H^k}{\partial u^k}(t)+\langle N(\widetilde{u}^k(t)),\nabla_{N}H^k(t)\rangle dt\Big|\mathfrak{F}_{\tau_k}\right]\\
&=\mathbb{E}\left[\int_{\tau_k}^{\tau_{k+1}}\mathbb{E}\left[\left(\widetilde{u}^k(t)\frac{\partial H^k}{\partial u^k}(t)+\widetilde{u}^k(t)M(\nabla_{N}H^k(t))\right)\mathbb{I}_{\tau_k\leq t<\tau_{k+1}}\Big|\mathcal{S}^k_t\right]dt\Big|\mathfrak{F}_{\tau_k}\right]\\
&=\mathbb{E}\left[\int_{\tau_k}^{\tau_{k+1}}\mathbb{E}\left[\left(\frac{\partial H^k}{\partial u^k}(t)+M(\nabla_{N}H^k(t))\right)\mathbb{I}_{\tau_k\leq t<\tau_{k+1}}\Big|\mathcal{S}^k_t\right]\widetilde{u}^k(t)dt\Big|\mathfrak{F}_{\tau_k}\right].
\end{align*}
This implies that $\widehat{u}^k$ is the optimal control to Subproblem \ref{spb}.
\end{proof}

\subsection{A Necessary Maximum Principle}
We now proceed to study the necessary version of the maximum principle for subproblems.
\begin{assumption}\label{ass nece}
For each $t_0\in[\tau_k,\tau_{k+1}]$ with $k\in \mathcal{I}_0$ and all bounded $\mathcal{S}_{t_0}$-measurable random variable $\eta$, the process
$\vartheta(t)=\eta\mathbb{I}_{[t_0,\tau_{k+1})}(t)$ belongs to $\mathcal{U}_k^{ad}$.
\end{assumption}
Thanking to the definition of $\mathcal{U}_k^{ad}$ with $k\in \mathcal{I}_0$, there exists a constant $\epsilon_0> 0$ such that
$$
u^k_{\epsilon}=\widehat{u}^k+\epsilon u^k\in \mathcal{U}_k^{ad},\;\epsilon\in[-\epsilon_0,\epsilon_0]
$$
for any $u^k\in \mathcal{U}_k^{ad}$ and bounded $\widehat{u}^k\in \mathcal{U}_k^{ad}$.
To simplify notations, for $\pi\in\{b^k_{\gamma},\sigma^k,h^k,f^k\}$ with $k\in \mathcal{I}_0$, we write
$$
(\nabla \pi(t))^T(x^k(t),u^k(t))=\frac{\partial \pi}{\partial x^k}(t)x^k(t)+\langle\nabla_M \pi(t),M(x^k(t))\rangle+\frac{\partial \pi}{\partial u^k}(t)u^k(t)+\langle\nabla_N \pi(t),N(u^k(t))\rangle.
$$

Next we consider the process $Z^k(t)$ obtained by differentiating $X^k_{\epsilon}(t)$ with respect to $\epsilon$ at $\epsilon=0$. Clearly, $Z^k(t)$ satisfies the following equation ($k\in \mathcal{I}_0$):
\begin{equation}\label{SDE nece}
\begin{cases}
dZ^k(t)=(\nabla b^k_{\gamma}(t))^T(Z^k(t),u^k(t))dt+(\nabla \sigma^k(t))^T(Z^k(t),u^k(t))dB_t\\
\qquad\qquad\;+(\nabla h^k(t))^T(Z^k(t),u^k(t))dA^{k+1}_t, \quad t\in[\tau_k,\tau_{k+1}].\\
Z^k(\tau_k)=0.
\end{cases}
\end{equation}

Then, the necessary maximum principle is given as follows.
\begin{thm}\label{thm nece}
Suppose that Assumption \ref{ass nece} holds. Then the following equalities are equivalent.
\begin{enumerate}[($\romannumeral1$)]
\item For all $u^k \in \mathcal{U}_k^{ad}$ with $k\in \mathcal{I}_0$,
\begin{equation}\label{nece1}
0=\frac{d}{d\epsilon}J(X^{k-1}(\tau_k),\mathfrak{L}^{k-1}(X^{k-1}(\tau_k)),\widehat{u}^k+\epsilon u^k)\Big|_{\epsilon=0}.
\end{equation}

\item For any $k\in \mathcal{I}_0$,
\begin{equation}\label{nece2}
0=\mathbb{E}\left[\frac{\partial H^k}{\partial u^k}(t)+M(\nabla_{N}H^k(t))\Big|\mathcal{S}^k_{t}\right]\Bigg|_{u^k=\widehat{u}^k},\quad \forall t\in(\tau_k,\tau_{k+1}].
\end{equation}
\end{enumerate}
\end{thm}

\begin{proof}
Assume \eqref{nece1} holds. Then
\begin{align}\label{7}
0&=\frac{d}{d\varepsilon}J(X^{k-1}(\tau_k),\mathfrak{L}^{k-1}(X^{k-1}(\tau_k)),\widehat{u}^k+\epsilon u^k)\Big|_{\epsilon=0}  \nonumber\\
&=\mathbb{E}\Big[\{\frac{\partial V^{k+1}}{\partial x^k}(X^k(\tau_{k+1}),\mathfrak{L}^{k}(X^k(\tau_{k+1})))+\nabla_{\mathfrak{L}^k}V^{k+1}(X^k(\tau_{k+1}),\mathfrak{L}^{k}(X^k(\tau_{k+1})))\}\cdot\widehat{Z}^k(\tau_{k+1})\nonumber\\
&+\int_{\tau_k}^{\tau_{k+1}}(\nabla f^k(t))^T(Z^k(t),u^k(t))dt\Big|\mathfrak{F}_{\tau_k}\Big],
\end{align}
where $\widehat{Z}^k(t)$ is the solution to \eqref{SDE nece}. By It\^{o}'s formula and \eqref{5}, we have
\begin{align}\label{6}
&\quad\,\mathbb{E}\left[\{\frac{\partial V^{k+1}}{\partial x^k}(X^k(\tau_{k+1}),\mathfrak{L}^{k}(X^k(\tau_{k+1})))+\nabla_{\mathfrak{L}^k}V^{k+1}(X^k(\tau_{k+1}),\mathfrak{L}^{k}(X^k(\tau_{k+1})))\}\cdot\widehat{Z}^k(\tau_{k+1})\Big|\mathfrak{F}_{\tau_k} \right]\nonumber\\
&=\mathbb{E}\left[\widehat{p}^k(\tau_{k+1})\widehat{Z}^k(\tau_{k+1})\Big|\mathfrak{F}_{\tau_k}\right]\nonumber\\
&=\mathbb{E}\left[\int_{\tau_k}^{\tau_{k+1}}\widehat{p}^k(t)d\widehat{Z}^k(t)+\widehat{Z}^k(t)d\widehat{p}^k(t)+d\langle\widehat{p}^k(t),\widehat{Z}^k(t)\rangle\Big|\mathfrak{F}_{\tau_k} \right]\nonumber\\
&=\mathbb{E}\Bigg[\int_{\tau_k}^{\tau_{k+1}}\widehat{p}^k(t)(\nabla b^k_{\gamma}(t))^T(\widehat{Z}^k(t),u^k(t))-\frac{\partial \widehat{H}^k}{\partial x^k}(t)\widehat{Z}^k(t)-\langle\nabla_{M}\widehat{H}^k(t),M(\widehat{Z}^k(t))\rangle\nonumber\\
&\quad\mbox{}+\widehat{q}(t)(\nabla \sigma^k(t))^T(\widehat{Z}^k(t),u^k(t))+\gamma^{\mathfrak{F}^{k+1}}(t)\widehat{r}(t)(\nabla h^k(t))^T(\widehat{Z}^k(t),u^k(t))dt\Big|\mathfrak{F}_{\tau_k}\Bigg]\nonumber\\
&=\mathbb{E}\left[\int_{\tau_k}^{\tau_{k+1}}\left(\frac{\partial \widehat{H}^k}{\partial u^k}(t)+M(\nabla_{N}\widehat{H}^k(t))\right)u^k(t)-(\nabla f^k(t))^T(\widehat{Z}^k(t),u^k(t))dt\Big|\mathfrak{F}_{\tau_k}\right].
\end{align}
Combining \eqref{7} and \eqref{6}, one has
\begin{equation*}
0=\frac{d}{d\varepsilon}J(X^{k-1}(\tau_k),\mathfrak{L}^{k-1}(X^{k-1}(\tau_k)),\widehat{u}^k+\epsilon u^k)\Big|_{\epsilon=0}=\mathbb{E}\left[\int_{\tau_k}^{\tau_{k+1}}\left(\frac{\partial \widehat{H}^k}{\partial u^k}(t)+M(\nabla_{N}\widehat{H}^k(t))\right)u^k(t)dt\Big|\mathfrak{F}_{\tau_k}\right].
\end{equation*}
Set $u^k(t)=\eta\mathbb{I}_{(t_0,\tau_{k+1}]}(t)\,(t_0\in[\tau_k,\tau_{k+1}])$, where $\eta$ is a bounded $\mathcal{S}_{t_0}$-measurable random variable. Then,
\begin{align*}
0&=\frac{d}{d\epsilon}J(X^{k-1}(\tau_k),\mathfrak{L}^{k-1}(X^{k-1}(\tau_k)),\widehat{u}^k+\epsilon u^k)\Big|_{\epsilon=0}\\
&=\mathbb{E}\left[\int_{t_0}^{\tau_{k+1}}\left(\frac{\partial \widehat{H}^k}{\partial u^k}(t)+M(\nabla_{N}\widehat{H}^k(t))\right)\eta dt\Big|\mathfrak{F}_{\tau_k}\right]\\
&=\mathbb{E}\left[\int_{t_0}^{\tau_{k+1}}\left(\frac{\partial \widehat{H}^k}{\partial u^k}(t)+M(\nabla_{N}\widehat{H}^k(t))\right)\eta dt\right].
\end{align*}
Differentiating with respect to $t_0$, we have
\begin{equation}\label{+1}
0=\mathbb{E}\left[\left(\frac{\partial \widehat{H}^k}{\partial u^k}(t_0)+M(\nabla_{N}\widehat{H}^k(t_0))\right)\eta\right]=\mathbb{E}\left[\mathbb{E}\left[\left(\frac{\partial \widehat{H}^k}{\partial u^k}(t_0)+M(\nabla_{N}\widehat{H}^k(t_0))\right)\Big|\mathcal{S}^k_{t_0}\right]\eta\right].
\end{equation}
Since \eqref{+1} holds for all such $\mathcal{S}^k_{t_0}$-measurable $\eta$, it follows that
$$
0=\mathbb{E}\left[\left(\frac{\partial \widehat{H}^k}{\partial u^k}(t_0)+M(\nabla_{N}\widehat{H}^k(t_0))\right)\Big|\mathcal{S}^k_{t_0}\right],\quad \forall t_0\in[\tau_k,\tau_{k+1}].
$$
The argument above is reversible. Thus, \eqref{nece1} and \eqref{nece2} are equivalent.
\end{proof}

\begin{remark}
It is worth mentioning that Lemma \ref{lm} plays an important role for deriving the maximum principles. However, it loses for the optimal control problems by mean-field type (in the sense of $\mathbb{E}[\cdot]$) SDEs. Therefore, our results differs from many existing works. Moreover, when there is no mean-field term $M(X^{k}(t))$ (resp. $N(u^{k}(t))$) in \eqref{SDE2} and the performance functional $J^{k}(X^{k-1}(\tau_k),\mathfrak{L}^{k-1}(X^{k-1}(\tau_k)),\widetilde{u}^k)$ (see \eqref{+3}), it is easy to see that Theorems \ref{sufficient} and \ref{thm nece} hold without requiring  $\mathcal{M}^k_{\tau_k}= \mathfrak{F}_{\tau_k}$ (resp. $\mathcal{G}^k_{\tau_k}= \mathfrak{F}_{\tau_k}$).
\end{remark}

\section{Solvability of MMFSDEs and MMFBSDEs}
In this section, we prove the existence and uniqueness of solutions to  MMFSDEs and MMFBSDEs, respectively, and apply the obtained results to solve an optimal investment problem.

\subsection{Solution to MMFSDEs}
Omitting the fixed control $u^k\in \mathcal{U}_k^{ad}$ ($k\in \mathcal{I}_0$), MMFSDE with general initial condition is of the form
\begin{equation}\label{SDE3}
\begin{cases}
dX(t)=b(t,X(t),M(X(t)))dt+\sigma(t,X(t),M(X(t)))dB_t+h(t,X(t),M(X(t)))d\mathbb{H}_t, \quad t\in[0,T],\\
X(0)=\beta.
\end{cases}
\end{equation}

We first give some necessary assumptions for both the coefficients and the initial condition of equation \eqref{SDE3} as follows.
\begin{assumption}\label{exi uni sde}\mbox{}
For any $k\in\mathcal{I}_0$, assume that
\begin{enumerate}[($\romannumeral1$)]
\item $\beta\in L^2_{0}(\mathbb{P})$;

\item The functions $b^k(\cdot,X^k,M(X^k)),\sigma^k(\cdot,X^k,M(X^k))$ are $\mathfrak{F}$-progressively measurable, and $h^k(\cdot,X^k,M(X^k))$ is $\mathfrak{F}$-predictable;

\item  $\mathbb{E}\left[\int_{0}^T|b^k(t,0,0)|^2+|\sigma^k(t,0,0)|^2+\gamma^{\mathfrak{F}_k}|h^k(t,0,0)|^2dt\right]<+\infty$;

\item \quad$|b^k(t,x_1,M_1)-b^k(t,x_2,M_2)|^2+|\sigma^k(t,x_1,M_1)-\sigma^k(t,x_2,M_2)|^2+|h^k(t,x_1,M_1)-h^k(t,x_2,M_2)|^2\\
\leq C\left(|x_1-x_2|^2+|{M}_1-{M}_2|_{\mathcal{P}_2(\mathbb{R})}^2\right).$
\end{enumerate}
\end{assumption}

\begin{thm}\label{MMFSDE}
Under Assumption \ref{exi uni sde}, there exists a unique solution $X(\cdot)\in L_{T}^{\mathfrak{F}}$ to \eqref{SDE3}.
\end{thm}

\begin{proof}
Define a map $\Psi$ as follows:
\begin{align*}
\Psi(X(t))&=\beta+\int_{0}^{t}b(s,X(s),M(X(s)))ds+\int_{0}^{t}\sigma(s,X(s),M(X(s)))dB_s+\int_{0}^{t}h(s,X(s),M(X(s)))d\mathbb{H}_s\\
&=\beta+\int_{0}^{t}b(s,X(s),M(X(s)))+\Upsilon(s)h(s,X(s),M(X(s)))ds\\
&\quad\mbox{}+\int_{0}^{t}\sigma(s,X(s),M(X(s)))dB_s+\int_{0}^{t}h(s,X(s),M(X(s)))d\mathbb{A}_s.
\end{align*}
Consider the sequence $\{X^{(m)}(t)\}_{m=1}^{+\infty}$ defined by $X^{(m+1)}(t)=\Psi(X^{(m)}(t))$ with $X^0(t)=\eta$. We firstly show the existence of the solution. To alleviate notations, we write $\overline{X}^{(m+1)}(t)={X}^{(m+1)}(t)-{X}^{(m)}(t)$ and $\pi^{(m)}(t)=\pi(t,X^{(m)}(t),M(X ^{(m)}(t)))\,(\pi\in\{b,\sigma,h\})$.
Following from $\gamma^{\mathfrak{F}^{k+1}}(t)\leq C$, Cauchy's inequality, Lemmas \ref{ito} and \ref{lm}, for $t\in [\tau_k,\tau_{k+1}]$ ($k\in \mathcal{I}_0$), we have
\begin{align}\label{9}
&\quad\,\mathbb{E}\Big[|\overline{X}^{k,(m+1)}(t)|^2\Big|\mathfrak{F}_{\tau_k}\Big]\nonumber\\
&=\mathbb{E}\Bigg[\Big[\int_{\tau_k}^{t}(b^{k,(m)}(s)-b^{k,(m-1)}(s))+\gamma^{\mathfrak{F}^{k+1}}(s)(h^{k,(m)}(s)-h^{k,(m-1)}(s))ds\nonumber\\
&\quad\mbox{}+\int_{\tau_k}^{t}\sigma^{k,(m)}(s)-\sigma^{k,(m-1)}(s)dB_s+\int_{\tau_k}^{t}h^{k,(m)}(s)-h^{k,(m-1)}(s)dA_s\Big]^2\Big|\mathfrak{F}_{\tau_k}\Bigg]\nonumber\\
&\leq 4\mathbb{E}\Big[\left(\int_{\tau_k}^{t}b^{k,(m)}(s)-b^{k,(m-1)}(s)ds\right)^2\Big|\mathfrak{F}_{\tau_k}\Big]+4\mathbb{E}\Big[\left(\int_{\tau_k}^{t}\left(\gamma^{\mathfrak{F}^{k+1}}(s)\right)^2(h^{k,(m)}(s)-h^{k,(m-1)}(s))ds\right)^2\Big|\mathfrak{F}_{\tau_k}\Big]\nonumber\\
&\quad\mbox{}+4\mathbb{E}\Big[\left(\int_{\tau_k}^{t}\sigma^{k,(m)}(s)-\sigma^{k,(m-1)}(s)dB_s\right)^2\Big|\mathfrak{F}_{\tau_k}\Big]+4\mathbb{E}\Big[\left(\int_{\tau_k}^{t}h^{k,(m)}(s)-h^{k,(m-1)}(s)dA_s\right)^2\Big|\mathfrak{F}_{\tau_k}\Big]\nonumber\\
&\leq 4\mathbb{E}\Big[(t-\tau_k)\int_{\tau_k}^{t}(b^{k,(m)}(s)-b^{k,(m-1)}(s))^2+(h^{k,(m)}(t)-h^{k,(m-1)}(s))^2ds\Big|\mathfrak{F}_{\tau_k}\Big]\nonumber\\
&\quad\mbox{}+4\mathbb{E}\Big[\int_{\tau_k}^{t}\left(\sigma^{k,(m)}(s)-\sigma^{k,(m-1)}(s)\right)^2 ds\Big|\mathfrak{F}_{\tau_k}\Big]+4\mathbb{E}\Big[\int_{\tau_k}^{t}\gamma^{\mathfrak{F}^{k+1}}(s)\left(h^{k,(m)}(s)-h^{k,(m-1)}(s)\right)^2 ds\Big|\mathfrak{F}_{\tau_k}\Big]\nonumber\\
&\leq C\mathbb{E}\Big[\int_{\tau_k}^{t}|\overline{X}^{k,(m)}(s)|^2ds\Big|\mathfrak{F}_{\tau_k}\Big]
+C\mathbb{E}\Big[\mathbb{E}\left[\int_{\tau_k}^{t}|\overline{X}^{k,(m)}(s)|^2ds\Big|\mathcal{M}_{\tau_k}\right]\Big|\mathfrak{F}_{\tau_k} \Big]\nonumber\\
&\leq C\mathbb{E}\Big[\int_{\tau_k}^{t}|\overline{X}^{k,(m)}(s)|^2ds\Big|\mathfrak{F}_{\tau_k}\Big].
\end{align}
Set $\overline{X}^{k,(m)}(s)=0$ for $s\in[0,\tau_k)$. Then, \eqref{9} yields
$$
\mathbb{E}\Big[|\overline{X}^{k,(m+1)}(t)|^2\Big]\leq C\mathbb{E}\Big[\int_{\tau_k}^{t}|\overline{X}^{k,(m)}(s)|^2ds\Big]= C\mathbb{E}\Big[\int_{0}^{t}|\overline{X}^{k,(m)}(s)|^2ds\Big].
$$
According to (iii) and (iv) in Assumption \ref{exi uni sde}, it is easy to show that, for $t\in [\tau_k,\tau_{k+1}]$ ($k\in \mathcal{I}_0$),
$$\mathbb{E}\Big[|\overline{X}^{k,(1)}(t)|^2\Big]\leq C(1+\mathbb{E}[|{X}^{k,(0)}(t)|^2])=C(1+\mathbb{E}(\beta^2))=C.$$
Again by the induction,  for $t\in [\tau_k,\tau_{k+1}]$ and $m\geq 1$, we have
$$
\mathbb{E}\Big[|\overline{X}^{k,(m+1)}(t)|^2\Big]\leq C^{m+1}\frac{t^{m}}{m!}.
$$
Integrating both sides from $\tau_k$ to $\tau_{k+1}$ and taking expectation, one has
$$
\mathbb{E}\left[\int_{\tau_k}^{\tau_{k+1}}|\overline{X}^{k,(m+1)}(t)|^2dt\right]\leq C^{m+1}\frac{\mathbb{E}\left[(\tau_{k+1}-\tau_k)\right]^{m+1}}{(m+1)!}\leq \frac{(CT)^{m+1}}{(m+1)!}.
$$
Thus,
\begin{align*}
\mathbb{E}\Big[\int_{0}^T|\overline{X}^{(m)}(t)|^2dt\Big]&=\mathbb{E}\Big[\int_{0}^T|\sum \limits_{k=0}^{n}\overline{X}^{k,(m)}(t)\mathbb{I}_{t\in(\tau_k,\tau_{k+1}]}(t)|^2dt\Big]\\
&\leq(n+1)\sum\limits_{k=0}^{n}\mathbb{E}\Big[\int_{\tau_k}^{\tau_{k+1}}|\overline{X}^{k,(m)}(t)|^2dt\Big]\\
&\leq n(n+1)\frac{(CT)^{m+1}}{(m+1)!}\mathbb{E}\left[1+\beta^2\right],
\end{align*}
which implies that $\{X^{(m)}(t)\}_{m=1}^{+\infty}$ forms a Cauchy sequence in $L_{T}^{\mathfrak{F}}$. Letting $\lim \limits_{m\rightarrow+\infty}X^{(m)}(t)=X(t)$, we have
\begin{align*}
0=&\lim \limits_{m\rightarrow+\infty}\|b(\cdot,X^{(m)}(\cdot),M(X^{(m)}(\cdot)))-b(\cdot,X(\cdot),M(X(\cdot)))\|_{L_{T}^{\mathfrak{F}}}\\
=&\lim \limits_{m\rightarrow+\infty}\|\sigma(\cdot,X^{(m)}(\cdot),M(X^{(m)}(\cdot)))-\sigma(\cdot,X(\cdot),M(X(\cdot)))\|_{L_{T}^{\mathfrak{F}}}\\
=&\lim \limits_{m\rightarrow+\infty}\|h(\cdot,X^{(m)}(\cdot),M(X^{(m)}(\cdot)))-h(\cdot,X(\cdot),M(X(\cdot)))\|_{L_{T}^{\mathfrak{F}}}.
\end{align*}
This shows that $X(\cdot)\in L_{T}^{\mathfrak{F}}$ is a solution to \eqref{SDE3}.

Next we prove the uniqueness. Suppose that there are two solutions $\mathcal{X}^{(1)}(\cdot)$ and $\mathcal{X}^{(2)}(\cdot)$ to \eqref{SDE3}. Similar to the proof of the existence for solutions to \eqref{SDE3}, we can show that, for any $k\in\mathcal{I}_0$,
$$
\mathbb{E}\Big[|\mathcal{X}^{k,(1)}(t)-\mathcal{X}^{k,(2)}(t)|^2\Big|\mathfrak{F}_{\tau_k}\Big]\leq C\int_{\tau_k}^{t}\mathbb{E}\Big[|\mathcal{X}^{k,(1)}(s)-\mathcal{X}^{k,(2)}(s)|^2\Big|\mathfrak{F}_{\tau_k}\Big]ds,\quad t\in [\tau_k,\tau_{k+1}].
$$
By Gronwall's inequality, we have $\mathbb{E}\left[|\mathcal{X}^{k,(1)}(t)-\mathcal{X}^{k,(2)}(t)|^2\right]=0$ and so the uniqueness follows immediately.
\end{proof}

\subsection{Solution to MMFBSDEs }
In this subsection, we prove the existence and uniqueness of solutions to the general MMFBSDEs of the following form
\begin{equation}\label{BSPDE}
\begin{cases}
dp(t)=-F\left(t,p(t),\mathfrak{M}(p(t)),q(t),\mathfrak{N}(q(t)),r_1(t),\cdots,r_n(t),\mathfrak{Q}(r_1(t)),\cdots,\mathfrak{Q}(r_n(t))\right)dt\\
\qquad\qquad+q(t)dB_t+\sum\limits_{k=1}^{n}r_k(t)dA^k_t,\quad t\in[0,T],\\
p(T)=\zeta,
\end{cases}
\end{equation}
where
$$
\mathfrak{M}(\cdot)=\mathfrak{N}(\cdot)=\mathfrak{Q}(\cdot)=\sum\limits_{k=1}^{n}\mathbb{E}\left[W^k(t)\times(\cdot)^k(t)\Big|\mathcal{M}^k_t\right]\mathbb{I}_{t\in (\tau_k,\tau_{k+1}]}+\mathbb{E}\left[W^0(t)\times(\cdot)^0(t)\Big|\mathcal{M}^0_t\right]\mathbb{I}_{t\in [0,\tau_1]}.
$$
Here $W(t)=(W^0(t),\cdots,W^n(t))$ is a given non-negative and upper bounded process, and $F$ is an $\mathfrak{F}$-progressively measurable $C^1$ function.

Firstly, we recall the predictable representation theorem (PRT).
\begin{lemma}\cite{Aksamit2017Enlargement}\label{PRT}
(Predictable Representation Theorem) Under Assumption \ref{AH}, for any martingale $Z(\cdot)\in H_{T}^{\mathfrak{F}}$, there exist $\mathfrak{F}$-predictable processes $\varphi_0\in H_{T}^{\mathfrak{F}}$ and $\varphi_k\in H_{T}^{\mathfrak{F},A^k}\;(k\in\mathcal{I})$ such that
\begin{equation}\label{decomposition}
Z(t)=Z(0)+\int_0^t\varphi_0(s)dB_s+\sum\limits_{k=1}^{n}\int_0^t\varphi_k(s)dA^k_s.
\end{equation}
\end{lemma}

We also need the following lemmas.

\begin{lemma}\label{uniqueness of the decomposition}
The decomposition \eqref{decomposition} is unique.
\end{lemma}

\begin{proof}
If there are $(\varphi_0,\varphi_1,\cdots,\varphi_n),(\psi_0,\psi_1,\cdots,\psi_n)\in H_{T}^{\mathfrak{F}}\times H_{T}^{\mathfrak{F},A^1}\times\cdots\times H_{T}^{\mathfrak{F},A^n}$ satisfying
$$
Z(t)=Z(0)+\int_0^t\varphi_0(s)dB_s+\sum\limits_{k=1}^{n}\int_0^t\varphi_k(s)dA^k_s=Z(0)+\int_0^t\psi_0(s)dB_s+\sum\limits_{k=1}^{n}\int_0^t\psi_k(s)dA^k_s, \quad t\in[0,T],
$$
then
$$\int_0^t\varphi_0(s)-\psi_0(s)dB_s+\sum\limits_{k=1}^{n}\int_0^t(\varphi_k(s)-\psi_k(s))dA^k_s=0$$
and so
\begin{align*}
0&=\mathbb{E}\left[\int_0^T\varphi_0(s)-\psi_0(s)dB_s+\sum\limits_{k=1}^{n}\int_0^T(\varphi_k(s)-\psi_k(s))dA^k_s\right]^2\\
&=\mathbb{E}\left[\int_0^{T}\varphi_0(s)-\psi_0(s)dB_s\right]^2+\sum\limits_{k=1}^{n}\mathbb{E}\left[\int_0^{T}(\varphi_k(s)-\psi_k(s))dA^k_s\right]^2\\
&=\mathbb{E}\left[\int_0^{T}|\varphi_0(s)-\psi_0(s)|^2ds\right]+\sum\limits_{k=1}^{n}\mathbb{E}\left[\int_0^{T}\gamma^{\mathfrak{F}_k}(s)|\varphi_k(s)-\psi_k(s)|^2 ds\right]\\
&=\|\varphi_0(\cdot)-\psi_0(\cdot)\|^2_{H_{T}^{\mathfrak{F}}}+\sum\limits_{k=1}^{n}\|\varphi_k(s)-\psi_k(s)\|^2_{H_{T}^{\mathfrak{F},A^k}}.
\end{align*}
This implies that $\varphi_k(s)=\psi_k(s)\;(k\in\mathcal{I}_0)$ a.s..
\end{proof}

\begin{lemma}\label{BSDE lemma}
For $F(\cdot)\in H_{T}^{\mathfrak{F}}$ and $\zeta\in L^2_{0}(\mathbb{P})$, the following MMFBSDE
\begin{equation*}
\begin{cases}
dp(t)=-F(t)dt+q(t)dB_t+\sum\limits_{k=1}^{n}r_k(t)dA^{k}_t,\quad t\in[0,T];\\
p(T)=\zeta
\end{cases}
\end{equation*}
has a unique solution $(p,q,r_1,\cdots,r_n)\in L_{T}^{\mathfrak{F}}\times H_{T}^{\mathfrak{F}}\times H_{T}^{\mathfrak{F},A^1}\times\cdots\times H_{T}^{\mathfrak{F},A^n}$.
\end{lemma}

\begin{proof}
Set $C(t)=\mathbb{E}\left[\zeta+\int_{0}^{T}F(t)dt\Big|\mathfrak{F}_t\right]$. Clearly, $C(t)$ is an $\mathfrak{F}_t$-martingale. Moreover, it follows from  Cauchy's inequality that
\begin{equation*}
C^2(t)=\mathbb{E}\left[\left[\zeta+\int_{0}^{T}F(t)dt\right]\Big|\mathfrak{F}_{t}\right]^2\leq \mathbb{E}\left[\left[\zeta+\int_{0}^{T}F(t)dt\right]^2\Big|\mathfrak{F}_{t}\right]
\leq 2\mathbb{E}\left[\zeta^2\Big|\mathfrak{F}_{t}\right]+2T\mathbb{E}\left[\int_{0}^{T}F^2(t)dt\Big|\mathfrak{F}_{t}\right].
\end{equation*}
This indicates $C(\cdot)\in H_{T}^{\mathfrak{F}}$. By Lemmas \ref{PRT} and \ref{uniqueness of the decomposition}, there exists a unique $(q,r_1,\cdots,r_n)\in H_{T}^{\mathfrak{F}}\times H_{T}^{\mathfrak{F},A^1}\times\cdots\times H_{T}^{\mathfrak{F},A^n}$ such that
$$
C(t)=C(0)+\int_{0}^{t}q(s)dB_s+\sum\limits_{k=1}^{n}\int_{0}^{t}r_k(s)dA^k_s
$$
and so
$$
C(T)-C(t)=\int_t^{T}q(s)dB_s+\sum\limits_{k=1}^{n}\int_{t}^{T}r_k(s)dA^k_s.
$$
Setting
\begin{equation}\label{+9}
p(t)=C(t)-\int_{0}^tF(s)ds=\mathbb{E}\left[\zeta+\int_t^{T}F(s)ds\Big|\mathfrak{F}_t\right],
\end{equation}
 we have $p(T)=\zeta$ and
\begin{equation}\label{+13}
p(T)-p(t)=\int_t^{T}q(s)dB_s+\sum\limits_{k=1}^{n}\int_{t}^{T}r_k(s)dA^k_s-\int_t^{T}F(s)ds.
\end{equation}
This ends the proof.
\end{proof}

To ensure the existence and uniqueness result, we also need the following assumption.

\begin{assumption}\label{exi and uni}
Suppose that the following assumptions hold:
\begin{enumerate}[($\romannumeral1$)]
\item $\zeta\in L^2_{T}(\mathbb{P})$;

\item $F(t,0,0,\cdots,0)\in H_{T}^{\mathfrak{F}}$;
\item  For any $t,p_1,q_1,r_1,p_2,q_2,r_2$, there is a constant $C>0$ such that
\begin{align*}
&\left|F(t,p,\mathfrak{M},q,\mathfrak{N},r_1,\mathfrak{Q}_1,\cdots,r_n,\mathfrak{Q}_n)
    -F(t,\widetilde{p},\widetilde{\mathfrak{M}},\widetilde{q},\widetilde{\mathfrak{N}},\widetilde{r}_1,\widetilde{\mathfrak{Q}}_1,\cdots,\widetilde{r}_n,
    \widetilde{\mathfrak{Q}}_n)\right|^2\\
\leq& C\bigg(|p-\widetilde{p}|^2+|\mathfrak{M}-\widetilde{\mathfrak{M}}|_{\mathcal{P}_2(\mathbb{R})}^2+|q-\widetilde{q}|^2+|\mathfrak{N}-\widetilde{\mathfrak{N}}|_{\mathcal{P}_2(\mathbb{R})}^2
+\sum\limits_{k=1}^{n}\sqrt{\gamma^{\mathfrak{F}_k}}\left(|r_k-\widetilde{r}_k|^2+|(\mathfrak{Q}_k-\widetilde{\mathfrak{Q}}_k|_{\mathcal{P}_2(\mathbb{R})}^2 \right)\bigg).
\end{align*}
\end{enumerate}
\end{assumption}

\begin{thm}\label{MMFBSDE}
Under Assumption \ref{exi and uni}, MMFBSDE \eqref{BSPDE} has a unique solution $(p,q,r_1,\cdots,r_n)\in L_{T}^{\mathfrak{F}}\times H_{T}^{\mathfrak{F}}\times H_{T}^{\mathfrak{F},A^1}\times\cdots\times H_{T}^{\mathfrak{F},A^n}$.
\end{thm}

\begin{proof}
We decompose the proof into two steps.

{\bf Step 1: Special case}

Assume that the driver $F$ is independent of $p$ and $\overline{p}$ such that
\begin{equation}\label{BSDE 1}
\begin{cases}
dp(t)=-F\left(t,q(t),\mathfrak{N}(q(t)),r_1(t),\cdots,r_n(t),\mathfrak{Q}(r_1(t)),\cdots,\mathfrak{Q}(r_n(t))\right)dt\\
\qquad\qquad+q(t)dB_t+\sum\limits_{k=1}^{n}r_k(t)dA^k_t,\quad t\in[0,T],\\
p(T)=\zeta.
\end{cases}
\end{equation}

We first prove the uniqueness and  existence of solutions to \eqref{BSDE 1}. By Lemma \ref{BSDE lemma}, for each $m\in \mathbb{N}$, there exists a unique solution $\left(p^{(m+1)},q^{(m+1)},r^{(m+1)}_1,\cdots,r^{(m+1)}_n\right)\in L_{T}^{\mathfrak{F}}\times H_{T}^{\mathfrak{F}}\times H_{T}^{\mathfrak{F},A^1}\times\cdots\times H_{T}^{\mathfrak{F},A^n}$ to the following MMFBSDE:

\begin{equation}\label{BSDE 2}
\begin{cases}
dp^{(m+1)}(t)=-F\left(t,q^{(m)}(t),\mathfrak{N}(q^{(m)}(t)),r^{(m)}_1(t),\cdots,r^{(m)}_n(t),\mathfrak{Q}(r^{(m)}_1(t)),\cdots,\mathfrak{Q}(r^{(m)}_n(t))\right)dt\\
\qquad\qquad\qquad+q^{(m+1)}(t)dB_t+\sum\limits_{k=1}^{n}r^{(m+1)}_k(t)dA^k_t,\quad t\in[0,T],\\
p^{(m+1)}(T)=\zeta,
\end{cases}
\end{equation}
where $q^{0}(t)=r^{0}(t)=0$ a.s. for all $t\in[0,T]$.

For simplicity, we write
\begin{align}\label{10}
&F^{(m)}(t)=F(t,q^{(m)}(t),\mathfrak{N}(q^{(m)}(t)),r^{(m)}_1(t),\cdots,r^{(m)}_n(t),\mathfrak{Q}(r^{(m)}_1(t)),\cdots,\mathfrak{Q}(r^{(m)}_n(t))),\nonumber\\
&L^{(m+1)}(t)=\mathbb{E}\left[\int_t^{T}|q^{(m+1)}(s)-q^{(m)}(s)|^2+\sum\limits_{k=1}^{n}\gamma^{\mathfrak{F}^{k}}(s)|r_k^{(m+1)}(s)-r_k^{(m)}(s)|^2ds\right].
\end{align}

Next we show that $\left(p^{(m)},q^{(m)},r_1^{(m)},\cdots,r_n^{(m)}\right)$ forms a Cauchy sequence. In fact, it follows from Lemma \ref{lm} that
$$
\mathbb{E}\left[\int_t^{T}|F^{(m)}(s)-F^{(m-1)}(s)|^2ds\right]\leq  CL^{(m)}(t).
$$
Applying Lemma \ref{ito} to $|p^{{(m+1)}}(t)-p^{(m)}(t)|^2$ and taking expectation, we have
\begin{align}\label{Ln}
&\quad\, \mathbb{E}\left[|p^{{(m+1)}}(t)-p^{(m)}(t)|^2\right]\nonumber\\
&=2\mathbb{E}\left[\int_t^{T}\left(p^{{(m+1)}}(s)-p^{(m)}(s)\right)\left(F^{(m)}(s)-F^{(m-1)}(s)\right)ds\right]-L^{(m+1)}(t)\nonumber\\
&\leq\mathbb{E}\left[\int_t^{T}\frac{1}{\rho}|p^{{(m+1)}}(s)-p^{(m)}(s)|^2+\rho |F^{(m)}(s)-F^{(m-1)}(s)|^2ds\right]-L^{(m+1)}(t)\nonumber\\
&\leq\frac{1}{\rho}\mathbb{E}\left[\int_t^{T}|p^{{(m+1)}}(s)-p^{(m)}(s)|^2ds\right]+\rho CL^{(m)}(t)-L^{(m+1)}(t),
\end{align}
where $\rho$ is any positive constant. Multiplying by $e^{\frac{1}{\rho} t}$ and integrating both sides in $[0,T]$, one has
\begin{equation*}
0\leq\mathbb{E}\left[\int_{0}^{T}|p^{{(m+1)}}(s)-p^{(m)}(s)|^2ds\right]\leq \rho C\int_{0}^{T}e^{\frac{1}{\rho}t}L^{(m)}(t)dt-\int_{0}^{T}e^{\frac{1}{\rho} t}L^{(m+1)}(t)dt.
\end{equation*}
By choosing $\rho=\frac{1}{2C}$, we have
\begin{equation}\label{120}
\int_{0}^{T}e^{2Ct}L^{(m)}(t)dt\leq \frac{1}{2^{m-1}}\int_{0}^{T}e^{2Ct}L^{(1)}(t)dt\leq\frac{C}{2^{m-1}}
\end{equation}
and
\begin{equation}\label{113}
\|p^{(m+1)}(\cdot)-p^{(m)}(\cdot)\|_{H_{T}^{\mathfrak{F}}}\leq \frac{1}{2}\int_{0}^{T}e^{\frac{1}{\rho}t}L^{(m)}(t)dt\leq\frac{C}{2^m},
\end{equation}
which implies $\lim \limits_{m\rightarrow +\infty}\|p^{(m+1)}(\cdot)-p^{(m)}(\cdot)\|_{H_{T}^{\mathfrak{F}}}=0$. Substituting \eqref{120} and \eqref{113} into \eqref{Ln}, one has
$$
\mathbb{E}\left[L^{(m+1)}(t)\right]\leq \frac{C}{2^m}+\frac{1}{2}\mathbb{E}\left[L^{(m)}(t)\right],
$$
which implies
$$
\mathbb{E}\left[L^{(m+1)}(t)\right]\leq \frac{C}{2^m}+\frac{1}{2^m}\mathbb{E}\left[L^{(0)}(t)\right].
$$
 Combining this and \eqref{10} obtains that,  for any $t\in[0,T]$,
\begin{equation}\label{11}
0=\lim \limits_{m\rightarrow +\infty}\left(\mathbb{E}\left[\int_t^{T}|q^{(m+1)}(s)-q^{(m)}(s)|^2ds\right]+\mathbb{E}\left[\int_t^{T}\sum\limits_{k=1}^{n}\gamma^{\mathfrak{F}^{k}}(s)|r_k^{(m+1)}(s)-r_k^{(m)}(s)|^2ds\right]\right).
\end{equation}
From \eqref{113} and \eqref{11}, we can see that $\left(p^{(m+1)},q^{(m+1)},r^{(m+1)}_1,\cdots,r^{(m+1)}_n\right)$ forms a Cauchy sequence in $L_{T}^{\mathfrak{F}}\times H_{T}^{\mathfrak{F}}\times H_{T}^{\mathfrak{F},A^1}\times\cdots\times H_{T}^{\mathfrak{F},A^n}$. Thus, we know that  $\left(p^{(m)},q^{(m)},r^{(m)}_1,\cdots,r^{(m)}_n\right)$ converges to a limit $(p,q,r_1,\cdots,r_n)\in L_{T}^{\mathfrak{F}}\times H_{T}^{\mathfrak{F}}\times H_{T}^{\mathfrak{F},A^1}\times\cdots\times H_{T}^{\mathfrak{F},A^n}$. Letting $m\rightarrow +\infty$ in \eqref{BSDE 1}, we conclude that $(p,q,r)$ is a solution to \eqref{BSDE 1}. This indicates the existence of solutions to \eqref{BSDE 1}.

Finally,  we proceed to show the uniqueness of solutions to \eqref{BSDE 1}. Let $(p,q,r(\cdot))$ and $(p^{(0)},q^{(0)},r^{(0)}(\cdot))$ be two solutions of \eqref{BSDE 1}. As the same arguments in {\bf Step 1},
we can show that
\begin{align*}
&\quad\;\mathbb{E}[|p(t)-p^{(0)}(t)|^2]-\frac{1}{\rho}\mathbb{E}\left[\int_t^{T}|p(s)-p^{(0)}(s)|^2ds\right]\\
&\leq(\rho C-1)\mathbb{E}\Big[\int_t^{T}|q(s)-q^{(0)}(s)|^2+\sum\limits_{k=1}^{n}\gamma^{\mathfrak{F}^{k}}(s)|r_k(s)-r_k^{(0)}(s)|^2ds\Big].
\end{align*}
By setting $\rho=\frac{1}{2C}$, one has
$$
\mathbb{E}[|p(t)-p^{(0)}(t)|^2]\leq  2C\mathbb{E}\left[\int_t^{T}|p(s)-p^{(0)}(s)|^2ds\right].
$$
By Gronwall's Lemma, we can deduce that $\mathbb{E}|p(t)-p^{(0)}(t)|^2=0$ and $p(t)=p^{(0)}(t)$ a.s. and so
\begin{align*}
-\frac{1}{2}\mathbb{E}\bigg[\int_t^{T}|q(s)-q^{(0)}(s)|^2+\sum\limits_{k=1}^{n}\gamma^{\mathfrak{F}^{k}}(s)|r_k(s)-r_k^{(0)}(s)|^2ds\bigg]\geq0,
\end{align*}
which implies that $q(t)=q^{(0)}(t)$ and $r_k(t)=r_k^{(0)}(t)\,(k\in\mathcal{I})$ a.s. and so the uniqueness follows.

{\bf Step 2: General case}

Consider the following iteration with general driver $F$
\begin{equation}\label{BSDE 3}
\begin{cases}
dp^{(m+1)}(t)=-F\Big(t,p^{(m)}(t),\mathfrak{M}(p^{(m)})(t),q^{(m+1)}(t),\mathfrak{N}(q^{(m+1)}(t)),r^{(m)}_1(t),\cdots,r^{(m)}_n(t),\\
\qquad\qquad\qquad \mathfrak{Q}(r^{(m)}_1(t)),\cdots,\mathfrak{Q}(r^{(m)}_n(t))\Big)dt+q^{(m+1)}(t)dB_t+\sum\limits_{k=1}^{n}r^{(m+1)}_k(t)dA^k_t,\quad t\in[0,T],\\
p^{(m+1)}(T)=\zeta,
\end{cases}
\end{equation}
where $p^0(t)=0$. It follows from {\bf Step 1} that, for each $m\in \mathbb{N}$, there exists a unique solution $$\left(p^{(m)},q^{(m)},r^{(m)}_1,\cdots,r^{(m)}_n\right)\in L_{T}^{\mathfrak{F}}\times H_{T}^{\mathfrak{F}}\times H_{T}^{\mathfrak{F},A^1}\times\cdots\times H_{T}^{\mathfrak{F},A^n}$$
satisfying the following inequality by choosing $\rho=\frac{1}{2C}$,
\begin{align*}
&\quad\,\mathbb{E}\left[|p^{(m+1)}(t)-p^{(m)}(t)|^2\right]+\frac{1}{2}L^{(m+1)}(t)\\
&\leq 2C\mathbb{E}\left[\int_t^{T}|p^{(m+1)}(s)-p^{(m)}(s)|^2ds\right]
+\frac{1}{2}\mathbb{E}\left[\int_t^{T}|p^{(m)}(s)-p^{(m-1)}(s)|^2ds\right].
\end{align*}
Hence,
$$
\mathbb{E}\left[|p^{(m+1)}(t)-p^{(m)}(t)|^2\right]\leq 2C\mathbb{E}\left[\int_t^{T}|p^{(m+1)}(s)-p^{(m)}(s)|^2ds\right]
+\frac{1}{2}\mathbb{E}\left[\int_t^{T}|p^{(m)}(s)-p^{(m-1)}(s)|^2ds\right].
$$
Multiplying by $e^{2C t}$ and integrating both sides in $[T_0,T]$ $(T_0\in[\tau_k,\tau_{k+1}])$, one has
\begin{align*}
\int_{T_0}^{T}\mathbb{E}|p^{(m+1)}(t)-p^{(m)}(t)|^2dt&\leq\frac{1}{2}\int_{T_0}^{T}e^{2Ct}\mathbb{E}\left[\int_t^{T}|p^{(m)}(s)-p^{(m-1)}(s)|^2ds\right]dt\\
 &\leq C\int_{T_0}^{T}\mathbb{E}\left[\int_t^T|p^{(m)}(s)-p^{(m-1)}(s)|^2ds\right]dt.
\end{align*}
Noticing that $\mathbb{E}\left[\int_t^{T}|p^{(1)}(s)-p^{(0)}(s)|^2ds\right]\leq C$, it follows from the induction that
$$
\int_{T_0}^{T}\mathbb{E}|p^{(m+1)}(t)-p^{(m)}(t)|^2dt\leq \frac{C^{m+1}(T-T_0)^m}{m!}\leq \frac{C^{m+1}T^m}{m!}.
$$
Moreover,
\begin{align*}
\mathbb{E}\left[L^{(m+1)}(T_0)\right]&=\mathbb{E}\left[\int_{T_0}^{T}|q^{(m+1)}(s)-q^{(m)}(s)|^2+\sum\limits_{k=1}^{n}\gamma^{\mathfrak{F}^{k}}(s)|r_k^{(m+1)}(s)-r_k^{(m)}(s)|^2ds\right]\\
&\leq 4C\mathbb{E}\left[\int_{T_0}^{T}|p^{(m+1)}(s)-p^{(m)}(s)|^2ds\right]
+\mathbb{E}\left[\int_{T_0}^{T}|p^{(m)}(s)-p^{(m-1)}(s)|^2ds\right]\\
&\leq \frac{4C^{m+2}T^m}{m!}+\frac{C^{m}T^{m-1}}{(m-1)!}.
\end{align*}
Setting $T_0=0$, it is easy to see that $\left(p^{(m)},q^{(m)},r^{(m)}_1,\cdots,r^{(m)}_n\right)$ forms a Cauchy sequence in $L_{T}^{\mathfrak{F}}\times H_{T}^{\mathfrak{F}}\times H_{T}^{\mathfrak{F},A^1}\times\cdots\times H_{T}^{\mathfrak{F},A^n}$. A similar argument as in {\bf Step 1} shows that $(p,q,r)=\lim \limits_{m\rightarrow +\infty}(p^{(m)},q^{(m)},r^{(m)})$ is the unique solution to \eqref{BSPDE}, which completes the proof.
\end{proof}
\subsection{An Optimal Investment Problem}
In this subsection, we consider an optimal investment problem, in which the wealth process $X(t)$ governed by the following MMFSDE:
\begin{equation}\label{SDE4}
\begin{cases}
dX^{(i)}(t)=X^{(i)}(t)\left[(S^{i,0}(t)-u^{(i)}(t)+\gamma^{(i+1)}(t)S^{i,2}(t))dt+S^{i,1}(t)dB_t+S^{i,2}(t)dA^{(i+1)}_t\right]\\
\qquad\qquad\:\: \mbox{}+\overline{X}^{(i)}(t)S^{i,3}(t)dt, \quad t\in(\tau_i,\tau_{i+1}],\\
X^{(i)}(\tau_i)=X^{(i-1)}(\tau_i),\quad X^{(-1)}(\tau_0)=x_0\in \mathbb{R}^+.
\end{cases}
\end{equation}
 Here $\overline{X}^{(i)}(t)=M(X^{(i)}(t))$ and
\begin{itemize}
\item $S^{i,j}(\cdot)\;(i=0,1,2,\;j=0,1,2,3)$ are all given deterministic, bounded and continuous functions on $[0,+\infty)$;
\item $S^{0,2}(t)>-1$, $S^{1,2}(t)>-1$, $S^{2,2}(t)=0$ and $S^{i,3}(t)\geq 0$ for all $t\geq0$;
\item $\gamma^{(i+1)}(t)=\gamma^{\mathfrak{F}^{i+1}}(t)$ is deterministic;
\item $
\mathcal{U}^{ad}=\Bigg\{u=(u^{(0)},u^{(1)},u^{(2)})\Big|u^{(i)}\,\mbox{is}\,\mathcal{O}(\mathcal{G}^{(i)},\mathbb{R})-\mbox{measurable with}\;{\sum \limits_{i=0}^{2}}\,\mathbb{E}\left[{\int_{\tau_i}^{\tau_{i+1}}}|u^{(i)}(t)|^2dt<\infty\right]\Bigg\};
$
\item $\mathcal{M}^i_{t}=\mathcal{G}^i_{t}=\mathcal{S}^i_{t}= \mathfrak{F}_{\tau_i},\; \forall t\in[\tau_i,\tau_{i+1})$,\;i=0,1,2.
\end{itemize}
We note that when $\gamma^{\mathfrak{F}^{i}}(t)$ ($i=1,2,3$) is deterministic and Assumption \ref{AH} holds, $\tau_i$ can be regarded as the default time in the Cox model \cite{Aksamit2017Enlargement}. Moreover, $\tau_i$ has a closed relationship with the Poisson process. If $\mathbb{N}^i_t$ is an inhomogeneous Poisson process with intensity $\gamma^{\mathfrak{F}^{i}}(t)$ and first jump time $\tau_i$, then $\mathbb{I}_{\tau_i\leq t}=\mathbb{N}^i_{t\wedge\tau_i}$, where $\mathbb{I}_{\tau_i\leq t}$ is the indicator process.

According to Theorem \ref{MMFSDE}, there exists a unique solution $X(\cdot)\in L_{T}^{\mathfrak{F}}$ to \eqref{SDE4} for any fixed $T>0$.  Taking conditional expectation on both sides in \eqref{SDE4} derives
\begin{equation*}
\begin{cases}
d\overline{X}^{(i)}(t)=\overline{X}^{(i)}(t)\left[S^{i,0}(t)-u^{(i)}(t)+\gamma^{(i+1)}(t)S^{i,2}(t)+S^{i,3}(t)\right]dt, \quad t\in(\tau_i,\tau_{i+1}],\\
\overline{X}^{(i)}(\tau_i)=X^{(i-1)}(\tau_i),\quad \overline{X}^{(-1)}(\tau_0)=x_0\in \mathbb{R}^+.
\end{cases}
\end{equation*}
Moreover, by comparison theorem, we know that $X^{(i)}(t)>0$ a.s.. Then, for any $t\in[\tau_{i},\tau_{i+1}]$ with $i=0,1,2$, considering the following SDE
\begin{equation*}
\begin{cases}
d{G}^{(i)}(t)={G}^{(i)}(t)\left[-S^{i,0}(t)+u^{(i)}(t)-\gamma^{(i+1)}(t)S^{i,2}(t)+|S^{i,1}(t)|^2+\gamma^{(i+1)}(t)|S^{i,2}(t)|^2\right]dt,\\
\qquad\qquad\mbox{}-{G}^{(i)}(t)S^{i,1}(t)dB_t-{G}^{(i)}(t)S^{i,2}(t)dA^{(i+1)}_t \quad t\in(\tau_i,\tau_{i+1}],\\
G^{(i)}(\tau_i)=G^{(i-1)}(\tau_i),\quad G^{(-1)}(\tau_0)=1.
\end{cases}
\end{equation*}
and applying Lemma \ref{ito} to ${X}^{(i)}(t){G}^{(i)}(t)$, one can easily check that
$$
d({X}^{(i)}(t){G}^{(i)}(t))=\overline{X}^{(i)}(t){G}^{(i)}(t)S^{i,3}(t)dt.
$$
Then applying Theorem A.1 in \cite{Peng2009BSDE} gives
\begin{align}
\ln X^{(i)}(t)
&=\ln X^{(i)}(\tau_i)+\ell(\tau_{i},t,u^{(i)})\nonumber\\
&=\ln X^{(i)}(\tau_i)-\ln {G}^{(i)}(t)+\ln\left[\int_{\tau_i}^{t}{G}^{(i)}(s)S^{i,3}(s)\varrho_{i}(s,u^{(i)})ds+1\right]\label{+42},
\end{align}
where
\begin{equation*}
\begin{cases}
{G}^{(i)}(t)=\exp\Bigg\{-{\int_{\tau_i}^{t}}S^{i,1}(s)dB_s-\int_{\tau_i}^{t}
S^{i,2}(s)dA_s+\int_{\tau_i}^{t}\left(
\ln (S^{i,2}(s)+1)-S^{i,2}(s)\right)d\mathbb{H}^k(s)\\
\qquad \qquad \qquad\mbox{}+{\int_{\tau_i}^{t}}\left[-S^{i,0}(s)+u^{(i)}(s)-\gamma^{(i+1)}(s)S^{i,2}(s)+\frac{1}{2}|S^{i,1}(s)|^2+\gamma^{(i+1)}(s)|S^{i,2}(s)|^2\right]ds\Bigg\},\\
\varrho_{i}(t,u^{(i)})=\exp\left\{{\int_{\tau_i}^{t}}S^{i,0}(s)-u^{(i)}(s)+\gamma^{(i+1)}(s)S^{i,2}(s)+S^{i,3}(s)ds\right\}.
\end{cases}
\end{equation*}
Next, we define a recursive utility process $\mathcal{P}(t)$ governed by the following MMFBSDE:
\begin{equation}\label{SDE5}
\begin{cases}
d\mathcal{P}^{(k)}(t)=-\Bigg[\nu_{k,0}(t)\mathbb{E}[\mathcal{P}^{(k)}(t) \Big|\mathfrak{F}_{\tau_k}]+\nu_{k,1}(t)\mathcal{Q}^{(k)}(t)+\nu_{k,2}(t)\Upsilon^{\mathfrak{F}}(t)\mathcal{R}^{(k)}(t)+\ln |X(t)u(t)|\Bigg]dt\\
\qquad\qquad \quad \mbox{} +\mathcal{Q}^{(k)}(t)dB_t+\mathcal{R}^{(k)}(t)dA^{(k+1)}_t,\quad t\in[\tau_k,\tau_{k+1}), \quad k=0,1,2,\\
\mathcal{P}^{(k)}(\tau_k)=\mathcal{P}^{(k-1)}(\tau_k), \quad k=1,2, \quad \mathcal{P}^{(2)}(T)= \ln X(T)
\end{cases}
\end{equation}
in the unknown process
$$
(\mathcal{P},\mathcal{Q},\mathcal{R})=\sum\limits_{k=0}^{1}(\mathcal{P}^{(k)},\mathcal{Q}^{(k)},\mathcal{R}^{(k)})(t)\mathbb{I}_{t\in [\tau_k,\tau_{k+1})}+(\mathcal{P}^{(2)},\mathcal{Q}^{(2)},\mathcal{R}^{(2)})(t)\mathbb{I}_{t\in [\tau_2,T]}.
$$
Here $\nu_{k,j}(\cdot)\;(k,j=0,1,2)$ are given deterministic and continuous functions on $[0,+\infty)$ satisfying $\nu_{k,0}(t)\geq0$, $\nu_{2,1}(t)>-1$ and $\nu_{2,2}(t)=0$. It follows from Theorem \ref{MMFBSDE} that \eqref{SDE5} has a unique solution $(\mathcal{P},\mathcal{Q},\mathcal{R})\in L_{T}^{\mathfrak{F}}\times H_{T}^{\mathfrak{F}}\times \bigcap \limits_{i=1}^2H_{T}^{\mathfrak{F},A^i}$.  Some related works concerned with the recursive utility process, we refer the reader to \cite{Duffie1992stochastic, Kraft2017Optimal, Li2020linear}.

Thus, the optimal investment problem can be described as follows.
\begin{prob}\label{exam pro}
Find the optimal investment strategy (optimal control) $\widehat{u}(t)=(\widehat{u}^0(t),\widehat{u}^1(t),\widehat{u}^2(t))\in \mathcal{U}^{ad}$ such that
$$
J(x_0,\mathfrak{L}(x_0),\widehat{u})=\sup \limits_{{u}\in \mathcal{U}^{ad}}J(x_0,\mathfrak{L}(x_0),u),
$$
where $J(x_0,\mathfrak{L}(x_0),u)=\mathbb{E}[\mathcal{P}(0)\Big|\mathfrak{F}_0]=\mathbb{E}[\mathcal{P}(0)]$.
\end{prob}

In order to solve Problem \ref{exam pro}, we first give the following proposition, which can be obtained from It\^{o}'s formula immediately.
\begin{prop}\label{+12}
The recursive utility can be rewritten as follows:
\begin{equation}\label{+11}
J(x_0,\mathfrak{L}(x_0),u)=\mathbb{E}[\mathcal{P}(0)]=\mathbb{E}\left[\int_0^TY(t)\ln |X(t)u(t)|dt+Y(T)\ln |X(T)|\right],
\end{equation}
where $Y(t)=(Y^{(0)}(t),Y^{(1)}(t),Y^{(2)}(t))$, and $Y^{(k)}(t)$ ($k=0,1,2$) are given by
\begin{equation}\label{e5-1}
\begin{cases}
dY^{(k)}(t)=\nu_{k,0}(t)\mathbb{E}[Y^{(k)}(t) \Big|\mathfrak{F}_{\tau_k}]dt+\nu_{k,1}(t)Y^{(k)}(t)dB_t+\nu_{k,2}(t)Y^{(k)}(t)dA^{(k+1)}_t,\quad t\in(\tau_k,\tau_{k+1}],\\
Y^{(k)}(\tau_k)=Y^{(k-1)}(\tau_k),\quad Y^{(-1)}(0)=1.
\end{cases}
\end{equation}
\end{prop}

\begin{remark}
For $k=0,1,2$, one can easily check that condition \eqref{++1} holds and by It\^{o}'s formula that
$$\mathbb{E}[Y^{(k)}(t) \Big|\mathfrak{F}_{\tau_k}]=Y^{(k)}(\tau_k)e^{\int_{\tau_k}^{t}\nu_{k,0}(s)ds}$$
and
\begin{equation}\label{+2}
Y^{(k)}(t)=\frac{Y^{(k)}(\tau_k)}{\varrho_{3,i}(t)}+Y^{(k)}(\tau_k)\int_{\tau_k}^{t}\varrho_{3,i}(s)e^{\int_{\tau_k}^{s}\nu_{k,0}(\varsigma)d\varsigma}ds>0 \quad a.s.,
\end{equation}
where
\begin{align*}
\varrho_{3,i}(t)=\exp\Big\{&-{\int_{\tau_k}^{t}}\nu^{k,1}(s)dB_s-\int_{\tau_k}^{t}\left(
\ln (\nu^{k,2}(s)+1)\right)dA_s\\
&\mbox{}+{\int_{\tau_k}^{t}}\left[\nu^{k,2}(s)\gamma^{(k+1)}(s)-\frac{1}{2}\left[\nu^{k,1}(s)\right]^2-\ln (\nu^{k,2}(s)+1)\gamma^{(k+1)}(s)\right]ds\Big\}.
\end{align*}
Moreover, if all coefficients in \eqref{SDE4} and \eqref{+11} are constants, then Proposition \ref{+12} indicates that Problem \ref{exam pro} reduces to Problem \ref{+8}.
\end{remark}
Then by Theorem \ref{multi}, we introduce the following subproblem.
\begin{subprob}\label{spb1}
Find an optimal investment strategy $\widehat{u}^{(2)}(\cdot)\in \mathcal{U}^{ad}_2$ such that
\begin{align*}
J^{(2)}(X^{(1)}(\tau_2),\mathfrak{L}^1(X^{(1)}(\tau_2)),\widetilde{u}^{(2)})&=\mathop{\esssup}\limits_{{u}^{(2)}\in \mathcal{U}^{ad}_2}J(X^{(1)}(\tau_2),\mathfrak{L}^1(X^{(1)}(\tau_2)),u^{(2)})\\
&=\mathop{\esssup}\limits_{{u^{(2)}}\in \mathcal{U}^{ad}_2}\mathbb{E}\left[\int_{\tau_2}^TY(t)\ln |X^{(2)}(t)u^{(2)}(t)|dt+Y(T)\ln |X^{(2)}(T)|\Big|\mathfrak{F}_{\tau_2}\right],
\end{align*}
where $X^{(2)}(T)=X(T)$.
\end{subprob}

Obviously, the Hamiltonian functional to Subroblem \ref{spb1} reads
$$
H^{(2)}=Y\ln |X^{(2)}u^{(2)}|+\left[X^{(2)}(S^{2,0}-u^{(2)})+\overline{X}^{(2)}S^{2,3}\right]p^{(2)}+S^{2,1}X^{(2)}q^{(2)},
$$
where $(p^{(2)},q^{(2)},r^{(2)})$ is given by
\begin{equation}\label{+7}
\begin{cases}
dp^{(2)}(t)=-[\frac{Y(t)}{X^{(2)}(t)}+(S^{2,0}(t)-u^{(2)}(t))p^{(2)}(t)+S^{2,3}(t)\overline{p}^{(2)}(t)+S^{2,1}(t)q^{(2)}(t)]dt+q^{(2)}(t)dB_t,\quad t\in[\tau_2,T],\\
p^{(2)}(T)=\frac{Y(T)}{X(T)}.
\end{cases}
\end{equation}
It follows from Theorems \ref{sufficient} and \ref{thm nece} that
$$
\mathbb{E}\left[\frac{Y(t)}{\widetilde{u}^{(2)}(t)}-\widetilde{X}^{(2)}(t)\widetilde{p}^{(2)}(t)\Big|\mathfrak{F}_{\tau_2}\right]=0,
$$
where $\widetilde{X}^{(2)}(t)$ and $\widetilde{p}^{(2)}(t)$ are respectively the solutions to \eqref{SDE4} and \eqref{+7} under the optimal investment strategy
$$
\widetilde{u}^{(2)}(t)=\frac{\mathbb{E}\left[Y(t)\Big|\mathfrak{F}_{\tau_2}\right]}{\mathbb{E}\left[\widetilde{X}^{(2)}(t)\widetilde{p}^{(2)}(t)\Big|\mathfrak{F}_{\tau_2}\right]},\quad t\in[\tau_2,T].
$$

On the other hand, by It\^{o}'s formula, we have
$$
\widetilde{X}^{(2)}(t)\widetilde{p}^{(2)}(t)=\mathbb{E}\left[\widetilde{X}^{(2)}(T)\widetilde{p}^{(2)}(T)+\int_t^TY(s)ds\Big|\mathfrak{F}_t\right]=\mathbb{E}\left[Y(T)+\int_t^TY(s)ds\Big|\mathfrak{F}_t\right]
$$
and so
\begin{equation*}
\begin{cases}
\widetilde{u}^{(2)}(t)=\frac{\mathbb{E}\left[Y(t)\Big|\mathfrak{F}_{\tau_2}\right]}{\mathbb{E}\left[\mathbb{E}\left[Y(T)+{\int_t^T}Y(s)ds\Big|\mathfrak{F}_t\right]\Big|\mathfrak{F}_{\tau_2}\right]}
=\frac{\mathbb{E}\left[Y(t)\Big|\mathfrak{F}_{\tau_2}\right]}{\mathbb{E}\left[Y(T)+{\int_t^T}Y(s)ds\Big|\mathfrak{F}_{\tau_2}\right]},\quad t\in[\tau_2,T]\\
J^{(2)}(X^{(1)}(\tau_2),\mathfrak{L}^{1}(X^{(1)}(\tau_2)),\widetilde{u}^{(2)})=\mathbb{E}\left[{\int_{\tau_2}^T}Y(t)[\ln |\widetilde{u}^{(2)}(t)|+\ell(\tau_{2},t,\widetilde{u}^{(2)})]dt+Y(T)\ell(\tau_{2},T,\widetilde{u}^{(2)})\Big|\mathfrak{F}_{\tau_2}\right]\\
\qquad\qquad \qquad\qquad \quad\; \mbox{}+\mathbb{E}\left[Y(T)+{\int_{\tau_2}^T}Y(s)ds\Big|\mathfrak{F}_{\tau_2}\right]\ln X^{(1)}(\tau_2).
\end{cases}
\end{equation*}

Similarly, $\widetilde{u}^{(1)}(t)$, $\widetilde{u}^{(0)}(t)$ and related utility can be given as follows:
\begin{equation*}
\begin{cases}
\widetilde{u}^{(1)}(t)=\frac{\mathbb{E}\left[Y(t)\Big|\mathfrak{F}_{\tau_1}\right]}{\mathbb{E}\left[Y(T)+{\int_{t}^T}Y(s)ds\Big|\mathfrak{F}_{\tau_1}\right]},\quad t\in[\tau_1,\tau_2),\\
\widetilde{u}^{(0)}(t)=\frac{\mathbb{E}\left[Y(t)\Big|\mathfrak{F}_{\tau_0}\right]}{\mathbb{E}\left[Y(T)+{\int_{t}^T}Y(s)ds\Big|\mathfrak{F}_{\tau_0}\right]},\quad t\in[0,\tau_1),\\
J^{(1)}(X^{(0)}(\tau_1),\mathfrak{L}^{0}(X^{(0)}(\tau_1)),\widetilde{u}^{(1)})={\sum \limits_{i=1}^{2}}\,\mathbb{E}\left[{\int_{\tau_i}^{\tau_{i+1}}}Y(t)[\ln |\widetilde{u}^{(i)}(t)|+\ell(\tau_{i},t,\widetilde{u}^{(i)})]dt+Y(T)\ell(\tau_{i},\tau_{i+1},\widetilde{u}^{(i)})\Big|\mathfrak{F}_{\tau_1}\right]\\
\qquad\qquad\qquad \qquad\quad\; \mbox{} +\mathbb{E}\left[Y(T)+{\int_{\tau_1}^T}Y(s)ds\Big|\mathfrak{F}_{\tau_1}\right]\ln X^{(0)}(\tau_1),\\
J^{(0)}(x_0,\mathfrak{L}(x_0),\widetilde{u}^{(0)})={\sum \limits_{i=0}^{2}}\,\mathbb{E}\left[{\int_{\tau_i}^{\tau_{i+1}}}Y(t)[\ln |\widetilde{u}^{(i)}(t)|+\ell(\tau_{i},t,\widetilde{u}^{(i)})]dt+Y(T)\ell(\tau_{i},\tau_{i+1},\widetilde{u}^{(i)})\right]\\
\qquad\qquad\qquad\quad \mbox{} +\mathbb{E}\left[Y(T)+{\int_{0}^T}Y(s)ds\right]\ln x_0.
\end{cases}
\end{equation*}

Finally, by Theorem \ref{multi}, $\widetilde{u}=(\widetilde{u}^{(0)},\widetilde{u}^{(1)},\widetilde{u}^{(2)})$ is indeed the global optimal investment strategy. Combining \eqref{+2}, we can see that
$$
\frac{\mathbb{E}\left[Y^{(k)}(t)\Big|\mathfrak{F}_{\tau_k}\right]}{\mathbb{E}\left[Y(T)+{\int_{t}^T}Y^{(k)}(s)ds\Big|\mathfrak{F}_{\tau_k}\right]}
=\frac{1}{e^{\int_{t}^{T}\nu_{k,0}(s)ds}+\int_t^Te^{\int_{t}^{s}\nu_{k,0}(\varsigma)d\varsigma}ds}
$$
is the unique feedback optimal control, which is non-zero and bounded. Therefore, we can organize the above results as the following theorem, which is unexpected but remarkably interesting.

\begin{thm}\label{+10}
The optimal investment strategy of Problem \ref{exam pro} is given by
\begin{equation}\label{+4}
\widehat{u}(t)=\frac{\mathbb{I}_{0\leq t<\tau_1}}{e^{\int_{t}^{T}\nu_{0,0}(s)ds}+\int_t^Te^{\int_{t}^{s}\nu_{0,0}(\varsigma)d\varsigma}ds}
+\frac{\mathbb{I}_{\tau_1\leq t<\tau_2}}{e^{\int_{t}^{T}\nu_{1,0}(s)ds}+\int_t^Te^{\int_{t}^{s}\nu_{1,0}(\varsigma)d\varsigma}ds}
+\frac{\mathbb{I}_{\tau_2\leq t\leq T}}{e^{\int_{t}^{T}\nu_{2,0}(s)ds}+\int_t^Te^{\int_{t}^{s}\nu_{2,0}(\varsigma)d\varsigma}ds}.
\end{equation}
\end{thm}

\section{Stability}
This section studies the mean square exponential stability and the almost sure exponential stability for the solution to MMFSDEs under optimal feedback control, respectively. To this end, we suppose that the coefficients in \eqref{SDE2} are all progressively measurable functions from $[0,+\infty)\times \Omega\times \mathbb{R}\times\mathcal{P}_2(\mathbb{R})\times \mathbb{R}\times\mathcal{P}_2(\mathbb{R})$ to $\mathbb{R}$, and that \eqref{SDE2} admits a unique solution $X(\cdot)\in L_{+\infty}^{\mathfrak{F}}$.
\subsection{Mean Square Exponential Stability}
Let us recall the following mean square exponential stability and almost sure exponential stability.
\begin{defn}
\begin{enumerate}[($\romannumeral1$)]
\item The state $X(t)$ is said to be mean square exponentially stable if there exists a pair of positive constants $M_0$ and $\alpha_0$ such
that $$\mathbb{E}\widetilde{}\left[(X(t))^2\right]\leq x_0^2\cdot M_0e^{-\alpha_0t}.$$
for all $t\geq0$ and $x_0\in \mathbb{R}$. Namely, $$\limsup\limits_{t\rightarrow{+\infty}}\frac{\ln \mathbb{E}\left[|X(t)|^2\right]}{t}<0$$ for all $x_0\in \mathbb{R}$.

\item The state $X(t)$ is said to almost surely exponentially stable if $$P\left\{\limsup\limits_{t\rightarrow{+\infty}}\frac{\ln|X(t)|}{t}<0\right\}=1$$
for all $x_0\in \mathbb{R}$.
\end{enumerate}
\end{defn}
Based on \eqref{SDE2}, consider the following system of MMFSDE with initial value $x_0\in\mathbb{R}$ and optimal feedback control $\hat{u}^k(t)=\phi^k(t,\hat{X}^k(t),M(\hat{X}^k)(t))\;(k\in\mathcal{I}_0)$:
\begin{equation*}
\begin{cases}
d\widehat{X}^{k}(t)=b_{\gamma}^{k}(t,\widehat{X}^{k}(t),M(\widehat{X}^{k}(t)),\widehat{u}^{k}(t),N(\widehat{u}^{k}(t)))dt+\sigma^{k}(t,\widehat{X}^{k}(t),M(\widehat{X}^{k}(t)),\widehat{u}^{k}(t),N(\widehat{u}^{k}(t)))dB_t\\
\qquad\qquad\; +h^{k}(t,\widehat{X}^{k}(t),M(\widehat{X}^{k}(t)),\widehat{u}^{k}(t),N(\widehat{u}^{k}(t)))dA^{k+1}_t, \quad t\in(\tau_k,\tau_{k+1}],\\
\widehat{X}^{k}(\tau_k)=\widehat{X}^{k-1}(\tau_k), \;\widehat{X}^{-1}(\tau_0)=x_0,
\end{cases}
\end{equation*}
where the $(n+1)$-th equation and the optimal feedback control remain valid on $t\in[T,+\infty)$. We have the following theorem for the mean square exponential stability.
\begin{thm}\label{stable}
Let
\begin{align*}
\;V_k(\widehat{X}^{k}(t),M(\widehat{X}^{k}(t)))=&2b_{\gamma}^{k}(t,\widehat{X}^{k}(t),M(\widehat{X}^{k}(t)),\widehat{u}^{k}(t),N(\widehat{u}^{k}(t)))\widehat{X}^{k}(t)\nonumber\\
&\mbox{} +|\sigma^{k}(t,\widehat{X}^{k}(t),M(\widehat{X}^{k}(t)),\widehat{u}^{k}(t),N(\widehat{u}^{k}(t)))|^2\nonumber\\
&\mbox{}+\gamma^{(i+1)}(t)|h^{k}(t,\widehat{X}^{k}(t),M(\widehat{X}^{k}(t)),\widehat{u}^{k}(t),N(\widehat{u}^{k}(t)))|^2.
\end{align*}
If there exist constants $C_0,C_1,\cdots,C_n$ such that $\sum\limits_{i=0}^{k}C_i>0$ and
\begin{equation}\label{++2}
\;\mathbb{E}\left[V_k(\widehat{X}^{k}(t),M(\widehat{X}^{k}(t)))\Big|\mathfrak{F}_{\tau_k}\right]
\leq -{C}_k\mathbb{E}\left[(\widehat{X}^{k}(t))^{2}\Big|\mathfrak{F}_{\tau_k}\right] \quad\mbox{a.s.}
\end{equation}
for all $k\in\mathcal{I}_0$ and $t\geq0$, then the state $\widehat{X}^{k}(t)$ is mean square exponentially stable.
\end{thm}
\begin{proof}
Applying It\^{o}'s formula to $e^{{C}_k(t-\tau_k)}(\widehat{X}^k(t))^{2}$ and taking  expectation, one can check that for any $t\in[\tau_k,\tau_{k+1}]$ or $t\in[\tau_n,+\infty)$,
\begin{align*}
&\quad\;\mathbb{E}\left[e^{{C}_k(t-T)}(\widehat{X}^k(t))^{2}\right]\\
&\leq\mathbb{E}\left[e^{{C}_k(t-\tau_k)}(\widehat{X}^k(t))^{2}\right]\\
&= e^0\mathbb{E}\left[(\widehat{X}^k(\tau_k))^{2}\right]
+\mathbb{E}\left[\int_{\tau_k}^{t}{C}_ke^{{C}_k(t-\tau_k)}(\widehat{X}^k(s))^{2}+e^{{C}_k(t-\tau_k)}V_k(\widehat{X}^{k}(t),M(\widehat{X}^{k}(t)))ds\right]\\
&=\mathbb{E}\left[(\widehat{X}^k(\tau_k))^{2}\right]
+\mathbb{E}\left[\int_{\tau_k}^{t}{C}_ke^{{C}_k(t-\tau_k)}(\widehat{X}^k(s))^{2}ds\right] +\mathbb{E}\left[\int_{\tau_k}^{t}e^{{C}_k(t-\tau_k)}\mathbb{E}\left[V_k(\widehat{X}^{k}(t),M(\widehat{X}^{k}(t)))\Big|\mathfrak{F}_{\tau_k}\right]ds\right]\\
&\leq \mathbb{E}\left[(\widehat{X}^k(\tau_k))^{2}\right]
+\mathbb{E}\left[\int_{\tau_k}^{t}{C}_ke^{{C}_k(t-\tau_k)}(\widehat{X}^k(s))^{2}ds\right] +\mathbb{E}\left[\int_{\tau_k}^{t}e^{{C}_k(t-\tau_k)}\mathbb{E}\left[-C_k(\widehat{X}^k(\tau_k))^{2}\Big|\mathfrak{F}_{\tau_k}\right]ds\right]\\
&=\mathbb{E}\left[(\widehat{X}^k(\tau_k))^{2}\right]
+\mathbb{E}\left[\int_{\tau_k}^{t}{C}_ke^{{C}_k(t-\tau_k)}(\widehat{X}^k(s))^{2}ds\right] -\mathbb{E}\left[\int_{\tau_k}^{t}e^{{C}_k(t-\tau_k)}C_k(\widehat{X}^k(\tau_k))^{2}ds\right]\\
&=\mathbb{E}\left[(\widehat{X}^k(\tau_k))^{2}\right],
\end{align*}
which implies
\begin{equation*}
\mathbb{E}\left[(\widehat{X}^k(t))^{2}\right]
\leq e^{{C}_k(T-t)}\mathbb{E}\left[(\widehat{X}^k(\tau_k))^{2}\right]\leq e^{({C}_k+C_{k-1})(T-t)}\mathbb{E}\left[(\widehat{X}^{k-1}(\tau_{k-1}))^{2}\right].
\end{equation*}
Thus, by the backward induction, we can obtain
\begin{equation*}
\mathbb{E}\left[(\widehat{X}^k(t))^{2}\right]
\leq x_0^2\cdot e^{(T-t)\sum\limits_{i=0}^{k}C_i},
\end{equation*}
which ends the proof with $M_0=e^{T\sum\limits_{i=0}^{k} C_i}$ and $\alpha_0=\sum\limits_{i=0}^{k}C_i$.
\end{proof}
The following proposition will be useful to verify condition \eqref{++2}.
\begin{prop}\label{+5}
For any given $t\in[\tau_k,\tau_{k+1}]$  or $t\in[\tau_n,+\infty)$, one has
$$\mathbb{E}\left[M(\widehat{X}^k(t))(\widehat{X}^k(t))\Big|\mathfrak{F}_{\tau_k}\right]\leq\mathbb{E}\left[(\widehat{X}^k(t))^{2}\Big|\mathfrak{F}_{\tau_k}\right] \quad\mbox{a.s.}.$$
\end{prop}
\begin{proof}
It follows from Cauchy's inequality and the condition $\mathcal{M}^k_{\tau_k}=\mathfrak{F}_{\tau_k}$ that
\begin{align*}
\mathbb{E}\left[M(\widehat{X}^k(t))(\widehat{X}^k(t))\Big|\mathfrak{F}_{\tau_k}\right]=&\mathbb{E}\left[\mathbb{E}\left[\widehat{X}^k(t)\Big|\mathcal{M}^k_t\right]
\cdot(\widehat{X}^k(t))\Big|\mathfrak{F}_{\tau_k}\right]\\
\leq&\mathbb{E}^{\frac{1}{2}}\left[\mathbb{E}^{2}\left[\widehat{X}^k(t)\Big|\mathcal{M}^k_t\right]\Big|\mathfrak{F}_{\tau_k}\right]
\cdot\mathbb{E}^{\frac{1}{2}}\left[(\widehat{X}^k(t))^{2}\Big|\mathfrak{F}_{\tau_k}\right]\\
\leq&\mathbb{E}^{\frac{1}{2}}\left[\mathbb{E}\left[(\widehat{X}^k(t))^{2}\Big|\mathcal{M}^k_t\right]\Big|\mathfrak{F}_{\tau_k}\right]
\cdot\mathbb{E}^{\frac{1}{2}}\left[(\widehat{X}^k(t))^{2}\Big|\mathfrak{F}_{\tau_k}\right]\\
=&\mathbb{E}^{\frac{1}{2}}\left[(\widehat{X}^k(t))^{2}\Big|\mathfrak{F}_{\tau_k}\right]\cdot\mathbb{E}^{\frac{1}{2}}\left[(\widehat{X}^k(t))^{2}\Big|\mathfrak{F}_{\tau_k}\right]\\
=&\mathbb{E}\left[(\widehat{X}^k(t))^{2}\Big|\mathfrak{F}_{\tau_k}\right]\quad\mbox{a.s.}
\end{align*}
and so the desired result is obtained.
\end{proof}

Next, concerning on the optimal controlled pair $(\widehat{X}(\cdot),\widehat{u}(\cdot))$ governed by \eqref{SDE4}, we further analyse the the mean square exponential stability of the state $\widehat{X}(t)$. Noticing that all the coefficients are bounded on $[0,+\infty)$, we denote
$$
\widehat{C}_i=\sup\limits_{t\in[0,+\infty)}\left[2S^{i,0}(t)-2\widehat{u}^{(i)}(t)+2\gamma^{(i+1)}(t)S^{i,2}(t)+2S^{i,3}(t)+|S^{i,1}(t)|^2+\gamma^{(i+1)}(t)|S^{i,2}(t)|^2\right].
$$
Furthermore, we need the following assumption.
\begin{assumption}\label{+18}
Suppose that $$\max\{\widehat{C}_1,\widehat{C}_1+\widehat{C}_2,\widehat{C}_1+\widehat{C}_2+\widehat{C}_3\}<0.$$
\end{assumption}
Then the result of mean square exponential stability for $\widehat{X}(t)$ can be specified as follows.
\begin{thm}\label{+19}
Under Assumption \ref{+18}, the state $\widehat{X}(t)$ is mean square exponentially stable.
\end{thm}
\begin{proof}
It follows from Proposition \ref{+5} and Theorem \ref{stable}  immediately.
\end{proof}
\begin{remark}
Theorem \ref{+19} implies that $\mathbb{E}\left[\int_0^{+\infty}e^{-\delta t}|X(t)|^2dt\right]<+\infty$ for all $\delta>\widehat{C}_1+\widehat{C}_2+\widehat{C}_3$.
\end{remark}
\subsection{Almost Sure Exponential Stability}
In general, it is hard to obtain the almost sure exponential stability for the solutions to nonlinear system \eqref{SDE1}. Thus this subsection will discuss the almost sure exponential stability for the solution to the linear MMFSDE \eqref{SDE4} under optimal feedback control. To this end, from \eqref{+42}, we focus on the state of the third stage ($i=2$) with
\begin{align}
\ln \widehat{X}^{(2)}(t)
&=\ln \widehat{X}^{(2)}(\tau_2)-\ln {\widehat{G}}^{(2)}(t)+\ln\left[\int_{\tau_2}^{t}{\widehat{G}}^{(i)}(s)S^{2,3}(s)\varrho_{i}(s,\widehat{u}^{(2)})ds+1\right]\nonumber\\
&=\ln \widehat{X}^{(2)}(\tau_2)-\ln {\widehat{G}}^{(2)}(t)\nonumber\\
&\quad\mbox{}+\ln\left[\int_{\tau_2}^{t}S^{2,3}(s)
\exp\left\{-{\int_{\tau_2}^{t}}S^{2,1}(s)dB_s+{\int_{\tau_2}^{s}}\left[\frac{1}{2}|S^{2,1}(\delta)|^2+S^{2,3}(\delta)\right]d\delta\right\}ds+1\right]\label{+16},
\end{align}
where
\begin{equation*}
\begin{cases}
{\widehat{G}}^{(2)}(t)=\exp\left\{-{\int_{\tau_2}^{t}}S^{2,1}(s)dB_s+{\int_{\tau_2}^{t}}\left[-S^{2,0}(s)+\widehat{u}^{(2)}(s)+\frac{1}{2}|S^{2,1}(s)|^2\right]ds\right\},\\
\varrho_{2}(t,\widehat{u}^{(2)})=\exp\left\{{\int_{\tau_2}^{t}}S^{2,0}(s)-\widehat{u}^{(2)}(s)+S^{2,3}(s)ds\right\}.
\end{cases}
\end{equation*}
Suppose that the optimal feedback control $\widehat{u}(t)$ given by \eqref{+4} remains valid on $t\in[T,+\infty)$, i.e.,
\begin{equation*}
\widehat{u}(t)=\frac{\mathbb{I}_{0\leq t<\tau_1}}{e^{\int_{t}^{T}\nu_{0,0}(s)ds}+\int_t^Te^{\int_{t}^{s}\nu_{0,0}(\varsigma)d\varsigma}ds}
+\frac{\mathbb{I}_{\tau_1\leq t<\tau_2}}{e^{\int_{t}^{T}\nu_{1,0}(s)ds}+\int_t^Te^{\int_{t}^{s}\nu_{1,0}(\varsigma)d\varsigma}ds}
+\frac{\mathbb{I}_{\tau_2\leq t\leq +\infty}}{e^{\int_{t}^{T}\nu_{2,0}(s)ds}+\int_t^Te^{\int_{t}^{s}\nu_{2,0}(\varsigma)d\varsigma}ds},
\end{equation*}
and the $(n+1)$-th equation in \eqref{SDE4} still holds for $t\in[T,+\infty)$. Then, one can easily deduce that \eqref{SDE4} has a unique solution ${\widehat{X}}(\cdot)\in{L_{+\infty}^{\mathfrak{F},\delta}}$ for some $\delta>0$ under the optimal feedback control $\widehat{u}(t)$. Moreover, we need the following assumption for the related coefficients.
\begin{assumption}\label{+17}
Suppose that the following assumptions hold:
\begin{enumerate}[($\romannumeral1$)]
\item The limit $\lim\limits_{t\rightarrow+\infty}S^{2,1}(t)$ exists;

\item $\lim\limits_{t\rightarrow+\infty}\left[S^{2,0}(t)-\widehat{u}^{(2)}(t)-\frac{1}{2}|S^{2,1}(t)|^2\right]<0$;

\item  $2S^{2,3}(s)\leq|S^{2,1}(s)|^2$, $\forall t\geq0$.
\end{enumerate}
\end{assumption}
Then we have the following result.
\begin{thm}\label{+22}
Under Assumption \ref{+17}, the state $\widehat{X}(t)$ is  almost surely exponentially stable.
\end{thm}
\begin{proof}
The proof is divided into two steps.

$\mathbf{Step 1}$ The quadratic variation of $\int_{\tau_2}^{t}S^{2,1}(s)dB_s$ reads $\int_{\tau_2}^{t}|S^{2,1}(s)|^2ds$. Since
$$
\lim\limits_{t\rightarrow+\infty}\frac{\int_{\tau_2}^{t}|S^{2,1}(s)|^2ds}{t}=\lim\limits_{t\rightarrow+\infty}|S^{2,1}(t)|^2<+\infty,\quad\mbox{a.s.,}
$$
it follows from Theorem 1.3.4 in \cite{MAO2008STOCHASTIC} that $$\lim\limits_{t\rightarrow+\infty}\frac{\int_{\tau_2}^{t}S^{2,1}(s)dB_s}{t}=0,\quad\mbox{a.s..}$$
Then combining
$$
\lim\limits_{t\rightarrow+\infty}\frac{{\int_{\tau_2}^{t}}\left[S^{2,0}(s)-\widehat{u}^{(2)}(s)-\frac{1}{2}|S^{2,1}(s)|^2\right]ds}{t}=
\lim\limits_{t\rightarrow+\infty}\left[S^{2,0}(t)-\widehat{u}^{(2)}(t)-\frac{1}{2}|S^{2,1}(t)|^2\right]<0\quad\mbox{a.s.,}
$$
we obtain that
\begin{equation}\label{+15}
\lim\limits_{t\rightarrow+\infty}-\ln {\widehat{G}}^{(2)}(t)<0, \quad\mbox{a.s..}
\end{equation}
$\mathbf{Step 2}$ Noticing that $\exp\left\{-2{\int_{\tau_2}^{t}}S^{2,1}(s)dB_s+2{\int_{\tau_2}^{t}}|S^{2,1}(s)|^2ds\right\}$ is an exponential martingale, one has
\begin{align*}
0&\leq\lim\limits_{t\rightarrow+\infty}\mathbb{E}\left[\frac{\ln^2\left[\int_{\tau_2}^{t}S^{2,3}(s)
\exp\left\{-{\int_{\tau_2}^{t}}S^{2,1}(s)dB_s+{\int_{\tau_2}^{t}}\left[\frac{1}{2}|S^{2,1}(s)|^2+S^{2,3}(s)\right]ds\right\}ds+1\right]}{t^2}\right]\\
&\leq\lim\limits_{t\rightarrow+\infty}\mathbb{E}\left[\frac{\left[\int_{\tau_2}^{t}S^{2,3}(s)
\exp\left\{-{\int_{\tau_2}^{t}}S^{2,1}(s)dB_s+{\int_{\tau_2}^{t}}\left[\frac{1}{2}|S^{2,1}(s)|^2+S^{2,3}(s)\right]ds\right\}ds\right]^2}{t^2}\right]\\
&\leq\lim\limits_{t\rightarrow+\infty}\mathbb{E}\left[\frac{\left[\int_{\tau_2}^{t}|S^{2,3}(s)|^2
\exp\left\{-{\int_{\tau_2}^{t}}2S^{2,1}(s)dB_s+{\int_{\tau_2}^{t}}\left[|S^{2,1}(s)|^2+2S^{2,3}(s)\right]ds\right\}ds\right]}{t^2}\right]\\
&\leq\lim\limits_{t\rightarrow+\infty}\mathbb{E}\left[\frac{\left[\int_{0}^{t}|S^{2,3}(s)|^2
\exp\left\{-{\int_{\tau_2}^{t}}2S^{2,1}(s)dB_s+2{\int_{\tau_2}^{t}}|S^{2,1}(s)|^2ds\right\}ds\right]}{t^2}\right]\\
&=\lim\limits_{t\rightarrow+\infty}\frac{\int_{0}^{t}|S^{2,3}(s)|^2ds}{t^2}\\
&=\lim\limits_{t\rightarrow+\infty}\frac{|S^{2,3}(t)|^2}{2t}\\
&=0,
\end{align*}
where we have used Cauchy's inequality. This implies that
\begin{equation}\label{+14}
\lim\limits_{t\rightarrow+\infty}\frac{\ln\left[\int_{\tau_2}^{t}S^{2,3}(s)
\exp\left\{-{\int_{\tau_2}^{t}}S^{2,1}(s)dB_s+{\int_{\tau_2}^{s}}\left[\frac{1}{2}|S^{2,1}(\delta)|^2+S^{2,3}(\delta)\right]d\delta\right\}ds+1\right]}{t}=0,\quad
\mbox{a.s..}
\end{equation}
Combining \eqref{+16}, \eqref{+15} and \eqref{+14}, we have
$$
\lim\limits_{t\rightarrow+\infty}\frac{\ln \widehat{X}^{(2)}(t)}{t}<0, \quad\mbox{a.s..}
$$
Thus, the state $\widehat{X}(t)$ is almost surely exponentially stable.
\end{proof}

As is well known, both the mean square exponential stability and the almost sure exponential stability imply the globally asymptotic stability \cite{deng1997stochastic}. Therefore, combining Theorems \ref{+19} and \ref{+22} we have the following result.
\begin{thm}
Under Assumptions \ref{+18} and \ref{+17},  the state $\widehat{X}(t)$ is globally asymptotically stable.
\end{thm}

\section{Conclusions}
This paper is devoted to the study of stochastic optimal control problems governed by a system of MMFSDEs. By deriving the backward induction formula, the original problem is decomposed into several subproblems. Then both sufficient and necessary maximum principles are given for these subproblems with random coefficients. Next, the existence and uniqueness results are obtained for both MMFSDEs and MMFBSDEs, and the obtained results are applied to solve an optimal investment problem. Finally, under suitable conditions,  the mean square exponential stability and almost sure exponential stability are guaranteed for the solutions to the MMFSDEs under optimal feedback controls.

We would like to mention that our results are still valid for solving the linear-quadratic optimal control problem governed by the MMFSDE. However, many interesting problems should be considered in the future: (i) find some second or higher-order necessary conditions for optimal control problems governed by MMFSDEs; (ii) solve  optimal control problems governed by MMFSDEs with time delays and/or regime switchings.


\end{document}